\documentclass[reqno,11pt]{amsart}
\usepackage{amsfonts, amsmath, amssymb, amsthm, color}
  \usepackage{amsfonts}
\usepackage{amssymb}
\usepackage{amsmath}
\usepackage{amsthm}
\usepackage{latexsym}
\usepackage{graphicx}
\usepackage{layout}
\usepackage{color}
\usepackage{geometry}

\allowdisplaybreaks[1]

\newtheorem{thm}{Theorem}[section]
\newtheorem{lemma}{Lemma}[section]
\newtheorem{prop}{Proposition}[section]

\newtheorem{rmk}{Remark}[section]

\theoremstyle{definition}
 \numberwithin{equation}{section}

\newcommand{\rr}{\mathbb{R}}
\newcommand{\al}{\alpha}
\newcommand{\de}{\delta}
 \newcommand{\eps}{\epsilon}
\newcommand{\la}{\lambda}

 \renewcommand{\(}{\left(}
\renewcommand{\)}{\right)}
\renewcommand{\[}{\left[}
\renewcommand{\]}{\right]}

\begin{document}
\title[New blow-up phenomena for  {\em SU(n+1)} Toda system]{New blow-up phenomena for {\em SU(n+1)} Toda system}
\author{Monica Musso}
\address[Monica Musso] {Departamento de Matem�ticas,
Pontificia Universidad Cat\'olica de Chile
Casilla 306, Correo 22 Santiago, Chile. }
\email{mmusso@mat.puc.cl}

\author{Angela Pistoia}
\address[Angela Pistoia] {Dipartimento SBAI, Universit\`{a} di Roma ``La Sapienza", via Antonio Scarpa 16, 00161 Roma, Italy}
\email{pistoia@dmmm.uniroma1.it}

\author{Juncheng Wei}
\address[Juncheng Wei] {Department of Mathematics,  The Chinese University of Hong Kong, Shatin, N.T., Hong Kong  and
Department of Mathematics, University of British Columbia, Vancouver, B.C., Canada, V6T 1Z2. }
\email{wei@math.cuhk.edu.hk}

\begin{abstract}
We consider the $SU(n+1)$ Toda system
$$\(S_\lambda\)\quad \left\{\begin{aligned}
& \Delta u_1+2\lambda e^{u_1}-\lambda e^{u_2}- \dots-\lambda e^{u_k}=0\quad \hbox{in}\ \Omega,\\
& \Delta u_2-\lambda e^{u_1}+2\lambda e^{u_2}- \dots-\lambda e^{u_k}=0\quad \hbox{in}\ \Omega,\\
&\vdots \hskip3truecm \ddots \hskip2truecm \vdots\\
& \Delta u_k -\lambda e^{u_1}-\lambda e^{u_2}- \dots+2\lambda e^{u_k}=0\quad \hbox{in}\ \Omega,\\
  &u_1=u_2=\dots=u_k=0\quad
\hbox{on}\ \partial\Omega.\\
\end{aligned}\right.
$$
 If  $0\in\Omega$  and $\Omega$ is symmetric with respect to the origin,    we construct a family of solutions $({u_1}_\la,\dots,{u_k}_\la)$
to $(S_\la )$ such that the $i-$th component ${u_i}_\la$ blows-up at the origin  with a mass $2^{i+1}\pi $ as $\la$ goes to zero.

 \end{abstract}

\date{\today}

 \maketitle

\section{Introduction}

Systems of elliptic equations in two dimensional spaces with exponential nonlinearity arise in many pure and applied disciplines such as  Physics, Geometry, Chemistry and Biology (see Chern and Wolfson \cite{CW}, Chipot, Shafrir and Wolansky \cite{CW}, Guest \cite{G} and Yang \cite{Y}).  Recently there is also considerable interest in the study of Toda-like systems, due to the importance in differential and algebraic geometry, and also mathematical physics.

We start with  the single component Liouville equation
\begin{equation}
\label{a1}
\Delta u+ \lambda e^{u}=0 \ \mbox{in} \ \Omega \subset \rr^2, \ \ u=0 \ \mbox{on} \ \partial \Omega
\end{equation}
which has been extensively studied  by many authors. Particular attention will be paid to the analysis of bubbling solutions.  Let $ (u_k, \lambda_k)$ be a bubbling sequence to (\ref{a1}), namely a family of solutions to  (\ref{a1}) with $ \lambda_k \int_{\Omega} e^{u_k} \leq C$, for some constant $C$, and $ \max_{x
 \in \Omega} u_k (x) \to +\infty$, as $k \to \infty$. Then it has been proved that all bubbles are simple, (see Brezis-Merle \cite{bm}, Nagasaki-Suzuki \cite{ns}, Li-Shafrir \cite{LS}), i.e. the local mass $\lim_{r \to 0} \lim_{k \to +\infty} \lambda_k \int_{B_r (x_k)} e^{u_k}$ equals  $ 8 \pi$ exactly. In fact in this case there is only one bubbling profile: after some rescaling, the bubble approaches to a solution of the Liouville equation
\begin{equation}
\label{a1.1}
\Delta w+e^w=0 \ \mbox{in} \ \rr^2, \int_{\rr^2} e^{w} <+\infty.
\end{equation}
\noindent
On the other hand, it is also possible to construct bubbling solutions with multiple concentrating points (see Baraket and Pacard \cite{bp}, del Pino, Kowalczyk  and Musso \cite{dkm}, Esposito, Grossi and Pistoia \cite{egp}). Degree formula has been obtained in Chen and Lin \cite{CL1,CL2}. Similar results can also be obtained when there are Dirac sources at the right hand side of (\ref{a1}) (see also  Bartolucci, Chen,  Lin and Tarantello  \cite{BCLT}).

Let us now turn to systems of Liouville type equations. In particular, we concentrate on the so-called $SU(3)$ Toda system
\begin{equation}
\label{a2}
\left\{\begin{array}{l}
\Delta u_1+ 2\lambda e^{u_1}-\lambda e^{u_2}=0 \ \mbox{in} \ \Omega, \\
\Delta u_2+ 2 \lambda e^{u_2}- \lambda e^{u_1}=0 \ \mbox{in} \ \Omega, \\
u_1=u_2 =0 \ \mbox{on} \ \partial \Omega.
\end{array}
\right.
\end{equation}

Systems of the above type (\ref{a2}) as well as its counterpart on a Riemannian surface $M$
\begin{align}\label{e104}
\left\{
\begin{array}{l}
\Delta u_1+2\rho_1(\frac{h_1e^{u_1}}{\int h_1e^{u_1}}-\frac{1}{|M|})-\rho_2(\frac{h_2e^{u_2}}{\int
 h_2e^{u_2}}-\frac{1}{|M|})=0\\
\Delta u_2+2\rho_2(\frac{h_2e^{u_2}}{\int h_2e^{u_2}}-\frac{1}{|M|})-\rho_1(\frac{h_1e^{u_1}}{\int
 h_1e^{u_1}}-\frac{1}{|M|})=0,
\end{array}
\right.
\end{align}
 arise from many different research areas in geometry and physics. In physics, it is related to
the relativistic version of non-abelian Chern-Simons models (see Dunne \cite{D}, Nolasco and Tarantello \cite{NT1}, Yang \cite{Y1}, Yang \cite{Y} and references therein). In geometry, the $SU(3)$ Toda system is closely related to
holomorphic curves (or harmonic sequence) of $M$ into $\mathbb{CP}^2$ (see   Bolton,  Jensen,   Rigoli  and  Woodward \cite{BJRW},  Chern and Wolfson \cite{CW},    Griffiths and Harris
 \cite{GH} and  Guest  \cite{G}). When $M=S^2$, it was proved that the solution
space of the $SU(3)$ Toda system is identical to the space of holomorphic curves of $S^2$ into $\mathbb{CP}^3$. We refer to  Lin,    Wei and Ye\cite{LWY} and the references therein.

For equation (\ref{a2}) or  \eqref{e104}, the first main issue is to determine the set of critical masses, i.e,  the limits of local massess
$ (\lambda_k \int_{B_r (x_k) } e^{u_{1,k}}, \lambda_k \int_{B_r (x)} e^{u_{2,k}})$
when $ u_{1,k} (x_k)=\max_{B_{r_0} (x_k)} \max (u_{1,k} (x), u_{2, k} (x)) \to +\infty $ and $ r_0>0$ is small radius.

In \cite{JLW} (see Lin,  Wei and Zhang \cite{LWZh} for another proof), Jost-Lin-Wang proved the following

\begin{thm}
\label{JLW}
Let $p_j$ be a bubbling point, i.e., $ \max_{B_{r_0} (p_j)} \max (u_{1, k} (x), u_{2, k} (x)) \to +\infty$ for some $r_0>0$. Define the local mass at $p_j$ as
 \begin{align}\label{e106}
\sigma_i(p_j)=\lim _{r\rightarrow 0}\lim _{k\rightarrow \infty} \lambda_k \int_{B_r (p_j)} e^{u_{i, k}}
\end{align}
 Then there are only
five possibility for $(\sigma_1,\sigma_2)$, i.e., $(\sigma _1,\sigma_2)$ could be one of $(4\pi ,0)$, $(0,4\pi)$, $(8\pi ,4\pi )$, $(4\pi, 8\pi)$ and $(8\pi,8\pi)$.
\end{thm}

Unlike single equations, according to Theorem \ref{JLW}, there are five possible blow-up scenarios.  A natural question is whether or not  {\em all} these blow-up scenarios are possible. Note that if we take $u=v$, this reduces to the single Liouville equation. By the construction in \cite{dkm} or in \cite{egp},  $(8\pi, 8\pi)$ is possible for any domain.

The last blow-up scenario is called {\em fully blow-up} case. The limiting equation becomes the $SU(3)$ Toda system in $\rr^2$
\begin{equation}
\label{su3}
\left\{\begin{array}{l}
\Delta w_1+2e^{w_1}-e^{w_2}=0 \ \mbox{in} \ \rr^2, \int_{\rr^2} e^{w_1} <+\infty, \\
\Delta w_2+2e^{w_2}-e^{w_1}=0 \ \mbox{in} \ \rr^2, \int_{\rr^2} e^{w_2} <+\infty
\end{array}
\right.
\end{equation}
whose solutions are completely characterized in Jost and Wang \cite{JW} and Lin, Wei and Ye \cite{LWY}. It is known that the masses are given by
$$ \int_{\rr^2} e^{w_1}=\int_{\rr^2} e^{w_2}=8 \pi.$$

The purpose of this paper is to show that the intermediate blow-up scenario does indeed occur. Namely for $SU(3)$ Toda system (\ref{a2}) in a symmetric domain (see definition below), we shall prove the existence of blowing-up solutions with local masses $(8\pi, 4 \pi)$ and $(4 \pi, 8 \pi)$. Note that there is no uniform limiting profile as in (\ref{su3}). Instead, both $u_1$ and $u_2$ have bubbles at the same place but with different blowing up rates and different limiting profiles (see remarks below).

In fact, more generally, we consider the $SU(n+1)$ Toda system
 \begin{equation}\label{p}
\left\{\begin{aligned}
& \Delta u_1+2\lambda e^{u_1}-\lambda e^{u_2}- \dots-\lambda e^{u_k}=0\quad \hbox{in}\ \Omega,\\
& \Delta u_2-\lambda e^{u_1}+2\lambda e^{u_2}- \dots-\lambda e^{u_k}=0\quad \hbox{in}\ \Omega,\\
&\vdots \hskip3truecm \ddots \hskip3truecm \vdots\\
& \Delta u_k -\lambda e^{u_1}-\lambda e^{u_2}- \dots+2\lambda e^{u_k}=0\quad \hbox{in}\ \Omega,\\
  &u_1=u_2=\dots=u_k=0\quad
\hbox{on}\ \partial\Omega,\\
\end{aligned}\right.
\end{equation}
where   $\Omega$ is a smooth bounded domain in $\rr^2$ and
$\lambda$ is a small positive parameter.

We will assume that $\Omega$ is $k-$symmetric, i.e.
\begin{equation}\label{ksym}
x\in\Omega\quad \hbox{if and only if}\quad   \Re_k x\in\Omega,\qquad \hbox{where}\quad \Re_k:=\(\begin{matrix}\cos{\pi\over k}&\sin{\pi\over k}\\
-\sin{\pi\over k}&\cos{\pi\over k}\\ \end{matrix}\).\end{equation}

The following is the main result of this paper.

 \begin{thm}\label{main}
 Assume that $\Omega$ is a $k-$symmetric domain(see \eqref{ksym}). If $\la$ is small enough problem \eqref{p} has a   solution $(u^1_\la,\dots,u^k_\la)$ such that
 $u^i_\la(x)=u^i_\la\(  \Re_k x\)$ for any $x\in\Omega.$ Moreover, it satisfies
  \begin{equation}\label{quantix}
 \lim\limits_{\la\to0}\la\int\limits_\Omega e^{u^i_\la(x)}dx= 2^{i+1}\pi,\quad i=1,\dots,k.
 \end{equation}

 \end{thm}

\begin{rmk}
{\it The symmetry condition (\ref{ksym}) is a technical condition. In the case of a general domain $\Omega$ with no symmetry, for $SU(3)$ Toda system with blow-up mass $(8 \pi, 4 \pi)$, $(4 \pi, 8 \pi)$, $(8 \pi, 8 \pi)$, some necessary conditions are needed.  For example, in the fully blowing-up case, there are six necessary conditions (see Lin,   Wei and Zhao \cite{LWZ1, LWZ2}). For our problem, in the case of a general domain with no symmetry, we expect that there should be at least four necessary conditions. }
\end{rmk}

\begin{rmk}
 
{\it 
As remarked earlier, there are no fully coupled limiting profile. For each $i=1,..., k$, after some scaling, $u_i$  has the following limiting profile
\begin{equation}
\label{1.10}
-\Delta w=|x|^{\alpha_i-2} e^w,  \ \mbox{in} \ \rr^2, \ \int_{\rr^2} |x|^{\alpha_i-2} e^w <\infty
\end{equation}
where $\alpha_i= 2^{i}$. Equation (\ref{1.10}) plays an important role in our construction. 
It is known that all solutions to (\ref{1.10}) have been classified  by Prajapat-Tarantello \cite{pt}. In fact solutions to (\ref{1.10}) are also nondegenerate--a key property that we shall use later  (see del Pino, Esposito and Musso  \cite{dem} and Lin, Wei and Ye \cite{LWY}). 
}

 \end{rmk}

\begin{rmk}
{\it
The construction we perform here is inspired by a recent result obtained by Grossi and Pistoia \cite{gp}, where they consider  the sinh-Poisson equation
\begin{equation}\label{sp}
-\Delta u=\la\sinh u\ \hbox{in}\ \Omega,\  v=0\ \hbox{on}\ \partial\Omega,\end{equation}
    $\Omega$ being a smooth  bounded domain in $\rr^2$  and
$\lambda$ being a small positive parameter. For any integer $k,$ Grossi and Pistoia \cite{gp} construct a family of solutions
to \eqref{sp} which blows-up at the origin as $\la\to0$ with positive and negative masses $4\pi k(k-1)$ and $4\pi k(k+1),$ respectively, provided    $0\in\Omega$  and $\Omega$ is symmetric with respect to the origin.
In particular, their result   gives a complete answer to an open problem formulated by Jost, Wang, Ye and Zhou in \cite{jwyz} similar to the one claimed in Theorem \ref{JLW}.
}

 \end{rmk}

\begin{rmk}
{\it In the case of $SU(3)$ Toda system, according to Lin,   Wei and Zhang \cite{LWZh}, there are two possible scenarios for the bubbling behavior $(8\pi, 4\pi)$. Theorem \ref{main} exhibits the first type. The second type is such that both $u_1$ and $u_2$  have the limiting profile (\ref{a1}). $u_1$ is the sum of two bubbles and $u_2$ has only one bubble. An open question is if the second type bubbling exists.  }
\end{rmk}

Let us comment on some recent related works. In \cite{LZ1, LZ2, LZ3}, Lin and Zhang studied general Liouville type systems with nonnegative coefficients. For Toda systems with singularities, the classification of local masses is given in Lin,   Wei and Zhang \cite{LWZh}. Sharp estimates for fully blow-up solutions for $SU(3)$ Toda system are given in  Lin,   Wei and Zhao \cite{LWZ1, LWZ2}. See also related studies by Malchiodi-Ndiaye \cite{MN}, Ohtsuka and  Suzuki \cite{OS}. As far as we know, Theorem \ref{main} seems to be the first existence result on  bubbling solutions to the $SU(3)$ Toda system.

\medskip

\noindent
{\bf Acknowledgment.}  Monica Musso
has been partly supported by Fondecyt Grant 1120151 and
CAPDE-Anillo ACT-125, Chile. Angela Pistoia has been supported by   ``Accordi Interuniversitari di Collaborazione Culturale e Scientifica Internazionale, A.F. 2012 between Universit\'a La Sapienza Roma and
  Pontificia Universidad Catolica de Chile".
Juncheng Wei was supported by a GRF grant from RGC of Hong Kong. We thank Professors Chang-Shou Lin and Lei Zhang for their interests in this work.

\section{The ansatz}\label{uno}

Let $\al\ge2.$ Let us introduce the functions
\begin{equation}\label{walfa}
w^\al_\de(x):=\ln 2\al^2{\de^\al\over\(\de^\al+|x|^\al\)^2}\quad
x\in\rr^2,\ \de>0
\end{equation}
which solve the singular
Liouville problem
\begin{equation}\label{plim}
-\Delta w=|x|^{\al-2}e^w\quad \hbox{in}\quad \rr^2,\qquad
\int\limits_{\rr^2} |x|^{\al-2}e^{w(x)}dx<+\infty.
\end{equation}

Functions $ w^\al_\de$ with suitable choices of $(\al, \de)$ constitute the main terms in the bubbling profiles of $u_i$.

Let us introduce the projection  $P  u$ of  a function $u$
into $H^1_0(\Omega),$ i.e.
\begin{equation}\label{pro}
 \Delta P u=\Delta u\quad \hbox{in}\ \Omega,\qquad  P u=0\quad \hbox{on}\ \partial\Omega.
\end{equation}

  We look for a  solution
to \eqref{p} as
\begin{equation}\label{ans}
\mathbf{u}_\la:=\({u_1}_\la,\dots,{u_k}_\la\)=\mathbf W_\la+ \boldsymbol{\phi}_\la,\ \end{equation}
with $\mathbf W_\la(x) :=\( W^1_\la,\dots, W^k_\la\)$ and $\boldsymbol\phi:=\({\phi_1}_\la,\dots,{\phi_k}_\la\)$.

Here for any $i=1,\dots,k$
\begin{equation}\label{answ}
 W^i_\la(x):= P w_i(x)-{1\over2}\sum \limits_{j=1\atop j\not=i}^kP w_i(x) , \ w_i(x):=w^{\al_i}_{\de_i}(x) \end{equation}
where \begin{equation}\label{alfa}
 \al_i=2^i
\end{equation}
and the concentration parameters satisfy
\begin{equation}\label{delta}
\de_i:=d_i \la^{2^{k-i}\over \al_i}=d_i \la^{2^{k-2i} }  \quad \hbox{for some}\ d_i>0.
\end{equation}
Let us point out that from (\ref{delta}) the following relations hold
\begin{equation}\label{delta0}
{\de_i\over \de_{i+1}}={d_i\over d_{i+1}} \la^{{3\over4}2^{k-2i} }  \quad \hbox{for}\ i=1,\dots,k-1.
\end{equation}
The rest term $\boldsymbol\phi\in  \mathbf{H}_k$ where (see \eqref{ksym})
\begin{equation}\label{hek}\mathbf{H}_k:=H_e ^k\quad \hbox{with}\ H_e:=\left\{\phi\in  \mathrm{H}^1_0(\Omega)\ :\ \phi(x)=\phi(\Re_kx)\ \hbox{for any}\ x\in\Omega \right\}.\end{equation}

\medskip

The choice of $\de_i$'s and $\al_i$'s    is motivated by the need for the interaction among bubbles to be small.
Indeed, an important feature is that each bubble interacts with   the other one and  in general the interaction is not negligible!
The interaction will be measured in Lemma \ref{errore} using the function $\boldsymbol\Theta:=(\Theta_1,\dots,\Theta_k)$ defined as
\begin{equation}\label{tetaj}
\Theta_j(y):=\(Pw_j-w_j-{1\over2}\sum\limits_{i=1\atop i\not=j}Pw_i\)(\de_j y)-(\al_j-2)\ln|\de_j
y| + \ln(2\la),\ j=1,\dots,k.\end{equation}
The choice of parameters $\al_j$ and $\de_j$   made in \eqref{alfa} and   \eqref{delta}   ensures that $\Theta_j$ is small.
 \medskip
In order to estimate $\Theta_j$ we need to
  introduce the  sets
  \begin{equation}\label{anelli}
 A_i:=\left\{x\in\Omega\ :\ \sqrt {\de_i\de_{i-1}}\le  |x|\le\sqrt {\de_i\de_{i+1}}\right\}  ,\ i=1,\dots,k,\end{equation}
where we set $\de_0:=0$ and $\de_{k+1}:=+\infty.$

We point out that if $j,\ell=1,\dots,k$
$${A_j\over\de_\ell}=\left\{y \in{\Omega\over\de_\ell}\ :\ {\sqrt {\de_{j-1}\de_j}\over\de_\ell}\le |y|\le{\sqrt {\de_j\de_{j+1}}\over\de_\ell}\right\}$$
and so {\em roughly speaking}
 ${A_j\over\de_\ell}$ shrinks to the origin if $\ell>j,$ ${A_j\over\de_j}$ invades the whole space $\rr^2$  and
 $ {A_j\over\de_\ell}$ runs off to infinity if $\ell<j.$

 More precisely, in order to have $\Theta_j$ small in Lemma \ref{teta} we will need to   choose   $\de_j$'s and $\al_j$'s so that
\begin{equation}\label{alfai}
-(\al_j-2)+ \sum\limits^k_{i=1\atop i<j}   \al_i=0
\end{equation}
and
\begin{equation}\label{deltai}
-  \al_j\ln\de_j +\sum\limits^k_{i=1\atop i>j}    \al_i
\ln\de_i-\ln(2\al_j^2) +h_j(0)-{1\over2}\sum\limits_{i=1\atop i\not=j}^k h_i(0)+\ln2\la=0,
\end{equation}
where we agree that if $j=1$ or $j=k$ the sum over the indices $i<j$ or $i>j$ is zero, respectively.
Here $h_i(x):=4\pi\al_i H(x,0).$
Moreover,
\begin{equation}\label{green}
G(x,y)={1\over 2\pi}\ln{1\over |x-y|}+H(x,y),\quad x,y\in\Omega
\end{equation}
is the Green's function of the Dirichlet Laplacian in $\Omega$ and $H(x,y)$ is its regular part.

By \eqref{alfai} we immediately deduce
\begin{equation}\label{alfai2}
  \al_1=2\quad\hbox{and}\quad \al_{j+1}=2\al_j\ \hbox{if}\ j=1,\dots,k-1
\end{equation}
and therefore \eqref{alfa}. Moreover,   by \eqref{deltai} we immediately
deduce that
\begin{equation}\label{delta2}\de_k^{\al_k} ={e^{h_k(0)-{1\over2}\sum\limits_{i=1\atop i\not=k}^k h_i(0)}\over \al_k^2}\la  \end{equation}
and
\begin{equation}\label{delta3}\de_{j }^{\al _{j }}={e^{h_j(0)-{1\over2}\sum\limits_{i=1\atop i\not=j}^k h_i(0)}\over \al_j^2}\de_{j+1}^{\al_{j+1}}\dots\de_k^{\al_k}\la\ \hbox{for}\ j=1,\dots,k-1,\end{equation}
which implies \eqref{delta}.

   By the maximum principle we easily deduce that
\begin{lemma}\label{pwi-lem}
\begin{align}\label{pro-exp}
 P  w_i(x)=& w_i(x)-\ln\(2\al_i^2\de_i^{\al_i}\)+h_i(x)+O\(\de_i^{\al_i}\)\nonumber\\ =&-2\ln\(\de_i^{\al_i}+|x|^{\al_i}\)+h_i(x)+O\(\de_i^{\al_i}\)
\end{align}
and for any $i,j=1,\dots,k$
\begin{equation}\label{pwi}
  Pw_i(\de_j y)=\left\{\begin{aligned}
  &-2\al_i\ln\(\de_j|y|\) +h_i(0)\\
  &\qquad +O\({1\over|y|^{\al_i}}\({\de_i\over\de_j}\)^{\al_i}\)+O\(\de_j|y|\)+O\(\de_i^{\al_i}\)&\ \hbox{if}\ i<j,\\
  & \\
& - 2\al_i\ln\de_i  -2\ln(1+|y|^{\al_i})+h_i(0)\\
  &\qquad+O\(\de_i|y|\)+O\(\de_i^{\al_i}\)&\ \hbox{if}\ i=j,\\
  & \\
&-2\al_i \ln\de_i+h_i(0)\\
  &\qquad+O\({ |y|^{\al_i}}\({\de_j\over\de_i}\)^{\al_i}\)+O\(\de_j|y|\)+O\(\de_i^{\al_i}\)&\ \hbox{if}\ i>j.\\
 \end{aligned}\right.\end{equation}
 Here $h_i(x):=4\pi\al_i H(x,0).$
\end{lemma}

Now, we are in position to prove the following crucial estimates.

\begin{lemma}\label{teta} Assume \eqref{alfa} and \eqref{delta}.
If    $j=1,\dots,k$ we have

\begin{equation}\label{er4}
\left|\Theta_j(y)\right| =O\(\delta_j|y|  +\la^{3\over2^k}\)\quad\hbox{for any}\  {y\in {A_j\over \de_j} }
 \end{equation}
and in particular
\begin{equation}\label{er4.1}
\sup\limits_{y\in {A_j\over \de_j} }\left|\Theta_j(y)\right| =O(1). \end{equation}

\end{lemma}
\begin{proof}

By Lemma \ref{pwi-lem}  (also using the mean value theorem $h_j \(\de_j|y|\)=h_j(0)+O\(\de_j|y|\)$), by \eqref{alfai} and by \eqref{deltai} we
deduce
\begin{align*}
  \Theta_j(y)   = &  \[ -\al_j\ln\de_j -\ln(2\al_j^2)+h_j(0)+O\(\de_j|y|\)+O\(\de_j^{\al_j}\)\]-(\al_j-2)\ln|\de_j y| \nonumber
\\   &
-{1\over2}\sum\limits_{i< j}  \[-2\al_i\ln\(\de_j|y|\)
+h_i(0)+O\({1\over|y|^{\al_i}}\({\de_i\over\de_j}\)^{\al_i}\)+O\(\de_j|y|\)+O\(\de_i^{\al_i}\)\]
\nonumber
\\   &-{1\over2}\sum\limits_{i> j} \[-2\al_i \ln\de_i+h_i(0)+O\({ |y|^{\al_i}}\({\de_j\over\de_i}\)^{\al_i}\)+O\(\de_j|y|\)+O\(\de_i^{\al_i}\)\]\nonumber
\\   &+\ln2\la \nonumber
\\
 = & \underbrace{ \[ -  \al_j\ln\de_j +\sum\limits_{i> j}   \al_i \ln\de_i-\ln(2\al_j^2)+h_j(0)-{1\over2}\sum\limits_{i=1\atop  i\not=j}^k h_i(0)+\ln2\la\]}_{=\ 0\ \hbox{because of}\ \eqref{deltai}}
 \nonumber
\\   &+\underbrace{\[-(\al_j-2)+ \sum\limits_{i< j}   \al_i\]}_{=\ 0\ \hbox{because of}\ \eqref{alfai}}\ln\(\de_j|y|\)\nonumber
\\  &+O\(\de_j|y|\)+\sum\limits_{i= 1}^kO\(\de_i^{\al_i}\)+\sum\limits_{i< j}O\({1\over|y|^{\al_i}}\({\de_i\over\de_j}\)^{\al_i}\)
+\sum\limits_{i> j}  O\({
|y|^{\al_i}}\({\de_j\over\de_i}\)^{\al_i}\) \nonumber
\\   = & O\(\de_j|y|\)+\sum\limits_{i= 1}^kO\(\de_i^{\al_i}\)+\sum\limits_{i< j}O\({1\over|y|^{\al_i}}\({\de_i\over\de_j}\)^{\al_i}\)
+\sum\limits_{i> j}  O\({
|y|^{\al_i}}\({\de_j\over\de_i}\)^{\al_i}\).
\end{align*}

 By \eqref{delta3} we  deduce that
$$ O\(\de_i^{\al_i}\)= O\(\la\)\ \hbox{because}\ 1\le i\le k.$$
Moreover,   if $y\in {A_j\over \de_j}$ then $\sqrt
{\de_{j-1}\over\de_j}\le |y|\le \sqrt {\de_{j+1}\over\de_j}$
and so
 if $j=2,\dots,k$ and $i<j$ we have
\begin{align*}
 O\({1\over|y|^{\al_i}}\({\de_i\over\de_j}\)^{\al_i}\)=&O\( \({\de_i^2\over\de_{j-1}\de_j}\)^
{\al_i\over2}\)=O\( \({\de_{j-1}\over\de_j}\)^{\al_i\over2}\)= O\(\la^{{3\over2}2^{k-2j+i}}\) =O\(\la^{3\over2^k}\),\end{align*}
 and if $j=1,\dots,k-1$ and $i>j$ we have
 \begin{align*}
O\({ |y|^{\al_i}}\({\de_j\over\de_i}\)^{\al_i}\)= &O\(
\({\de_{j+1}\de_j\over\de_i^2}\)^{\al_i\over2}\)= O\(
\({\de_j\over\de_{j+1}}\)^{\al_i\over2}\)=  O\(\la^{{3\over8}2^{k-2j+i}}\)=O\(\la^{3\over2}\). \end{align*}
We used \eqref{delta0} and \eqref{alfa}.
 Collecting all the previous estimates, we get \eqref{er4}.

 Estimate \eqref{er4.1} follows immediately by \eqref{er4}, because if $y \in {A_j\over \de_j}$ then
 $\delta_j |y|=O(1).$
\end{proof}

In the following, we will denote by
$$\|u\|_p:=\(\int\limits_\Omega |u(x)|^pdx\)^{1\over p}\quad \hbox{and}\quad\|u\| :=\(\int\limits_\Omega |\nabla u(x)|^2dx\)^{1\over 2}$$
the usual norms in the Banach spaces $\mathrm{L}^p(\Omega)$ and $\mathrm{H}^1_0(\Omega),$ respectively.
We also denote by $\mathbf u:=(u_1,\dots,u_k)\in\(\mathrm{H}^1_0(\Omega)\)^k$ and we set
$$\|\mathbf u\|_p=\sum\limits_{i=1}^k\|u_i\|_p\quad \hbox{and}\quad\|\mathbf u\| =\sum\limits_{i=1}^k\|u_i\|.$$

\section{Estimate of the error term}\label{due}

In this section we will estimate the following error term

\begin{align}\label{rla}
& \mathcal{R}_\la(x):=   \( R^1 _\la(x),\dots, R^k _\la(x) \),\
 x\in\Omega,\quad \hbox{where}\nonumber\\
 &R^j _\la(x):=-\Delta W^j_\la(x)-2\la e^{ W^j_\la(x)} +\la \sum\limits_{i=1\atop i\not=j}^ke^{ W^i_\la(x)} , \ j=1,\dots,k.
\end{align}

 \begin{lemma}\label{errore} Let $  \mathcal{R}_\la$ as in \eqref{rla}.
There exists $p_0>1$ and $\la_0>0$ such that for any $p\in(1,p_0 )$ and $\la\in(0,\la_0)$ we have
$$\| \mathcal{R}_\la\|_{p}= O\(\la^{{1\over2^k}{2-p\over p}}\).$$
\end{lemma}

\begin{proof}
 We will show that if  $p$ is close enough to 1
 \begin{align}\label{er1}
 &\left\| R^i _\la\right\|_p=O\(\la^{{1\over2^k}{2-p\over p}  }\),\ i=1,\dots,k .
\end{align}
The claim will follow.
By \eqref{answ}  we have
 \begin{align}\label{er1.1}
   R^i _\la &=-\Delta W^i_\la-2\la e^{ W^i_\la}  +\la \sum\limits_{j=1\atop j\not=i}^ke^ {W^j_\la} \nonumber\\
 &= -\Delta \(Pw_i-{1\over2}\sum\limits_{j=1\atop j\not=i }^kPw_j\)-2\la e^{Pw_i-{1\over2}\sum\limits_{j=1\atop j\not=i }^kPw_j}+\la \sum\limits_{j=1\atop j\not=i} ^k e^{Pw_j-{1\over2}\sum\limits_{\ell=1\atop \ell\not=j}^kPw_\ell} \nonumber\\
 &=|x|^{\alpha_i-2}e^{w_i(x)}-{1\over2}\sum\limits_{j=1\atop j\not=i}^k|x|^{\alpha_j-2}e^{w_j(x)}-2\la e^{Pw_i-{1\over2}\sum\limits_{j=1\atop j\not=i}^kPw_j}+\la \sum\limits_{j=1\atop j\not=i} ^ke^{Pw_j-{1\over2}\sum\limits_{\ell=1\atop \ell\not=j}^k Pw_\ell}    \end{align}

Therefore, by  \eqref{tetaj} we get
 \begin{align}\label{er1.2}
&\int\limits_\Omega| {R_i}_\la(x)|^pdx\  =O\( \sum\limits_{j=1 }^k\int\limits_{\Omega}\left||x|^{\alpha_j-2}e^{w_j(x)} -2\la e^{Pw_j-{1\over2}\sum\limits_{\ell\not=j}Pw_\ell}  \right|^pdx\) \nonumber
\\ &= O\(\sum\limits_{j=1}^k\int\limits_{A_j}\left||x|^{\alpha_j-2}e^{w_j(x)} - 2\la e^{Pw_j-{1\over2}\sum\limits_{\ell\not=j}Pw_\ell}  \right|^pdx\)
\nonumber
\\ &+
 O\(\sum\limits_{j=1}^k\sum\limits_{r =1\atop r\not=j}^k\int\limits_{A_r}\left||x|^{\alpha_j-2}e^{w_j(x)} \right|^pdx\)
+ O\(\sum\limits_{j=1}^k\sum\limits_{r =1\atop r\not=j}^k\int\limits_{A_r}\left|\la e^{Pw_j(x)-{1\over2}\sum\limits_{\ell\not=j}Pw_\ell   (x)} \right|^pdx\)
 \end{align}

\medskip
Let us estimate the first term in \eqref{er1.2}, which gives the rate of $\left\| R^i _\la\right\|_p$. For any $j=1,\dots,k$ we have
\begin{align}\label{er1.3}
&\int\limits_{A_j}\left||x|^{\alpha_j-2}e^{w_j(x)} -2 \la e^{Pw_j-{1\over2}\sum\limits_{\ell\not=j}Pw_\ell}  \right|^pdx\ (\hbox{we use \eqref{tetaj}})\nonumber
\\
&   =\int\limits_{A_j}\left||x|^{\alpha_j-2}e^{w_j(x)}\[1-  e^{\Theta_j\({x/\de_j}\)}\]\right|^pdx \ (\hbox{we set $x=\de_jy$})\nonumber
\\
&   = \de_j^{2-2p}\int\limits_{A_j\over \de_j}
\frac{|y|^{(\al_j-2)p}}{\(1+|y|^{\al_j}\)^{2p}} \left|1-
e^{\Theta_j\(y\)}
\right|^pdy \  \hbox{(we use that $e^t-1=e^{\eta t}t$  for some
$\eta\in(0,1) $ and  Lemma \ref{teta})}\nonumber
\\ &
=O\(\de_j^{2-2p}\int\limits_{A_j\over \de_j}\frac{|y|^{(\al_j-2)p}}{\(1+|y|^{\al_j}\)^{2p}}\left|
 \Theta_j(y)\right|^pdy\)=\nonumber
\\ &  =O\(\de_j^{2-2p}\int\limits_{A_j\over \de_j}\frac{|y|^{(\al_j-2)p}}{\(1+|y|^{\al_j}\)^{2p}}\left|
 \delta_j |y|+\lambda^{3\over2^k}\right|^pdy\)\  \hbox{(we use that   \eqref{delta})}\nonumber \\ &=
O\(\de_j^{2-2p}\lambda^{{3\over2^k}p}\)+O\(\de_j^{2- p}\)  \  \hbox{(we use that $ \de_1\le\de_j\le\de_k$)}\nonumber \\ & =
O\(\la^{2^{k-2}(1-p)+{3\over2^k}p}\)+O\(\la^{ {1\over2^k}(2-p) }\)=O\(\la^{ {1\over2^k}(2-p) }\).
  \end{align}

  \medskip
   Let us estimate the second term in  \eqref{er1.2}. For any $j=1,\dots,k$ and $r\not=j$ we have

\begin{align}\label{er6}
      &\int\limits_{A_r}\left|{ |x|^ {\alpha_j-2}\over(\de_j^{\al_j}+|x|^{\al_j})^2}\right|^pdx\quad  \hbox{(we scale $x=\delta_j y$)}\nonumber\\
  &=C\de_j^{2-2p} \int\limits_{{\sqrt{\de_{r-1}\de_r}\over \de_j}\le |y|\le {\sqrt{\de_r\de_{r+1}}\over \de_j}}
  {|y|^{(\al_j-2)p}\over \(1+|y|^{\al_j}\)^{2p}} dy\nonumber \\
   &=
  \left\{
  \begin{aligned}
&  O\(\de_j^{2-2p}\({ {\sqrt{\de_r\de_{r+1}}}\over \de_j^2}\)^{(\al_j-2)p+2}\)=
O\( \de_j^{2-2p} \({\de_r\over\de_{r+1}}\)^{(\al_j-2)p+2\over2} \)\\ &\qquad \hbox{if}\  r =1,\dots,k-1\ \hbox{and}\ j>r,\\
& \\
  &O\(\de_j^{2-2p}\({ \de_j\over\sqrt{\de_{r-1}\de_r} }\)^{-(\al_j+2)p+2}\)=O\(\de_j^{2-2p}\({  \de_{r-1}\over \de_r }\)
  ^{(\al_j+2)p-2\over2}\)\\ &\qquad \hbox{if}\  r=2,\dots,k \ \hbox{and}\  j<r.\\
  \end{aligned}
  \right.
  \nonumber\\
  & \quad\hbox{(we use \eqref{delta} and \eqref{delta0})}\nonumber\\ &  =
  \left\{
 \begin{aligned}
&  O\(\de_1^{2-2p}\( {\de_{j-1} \over \de_j }\)^{{\al_j\over2}-(p-1) }\)=
O\( \la^{-2^{k-1}(p-1)} \la ^{{3\over4} 2^{k-2j+2}\(2^{j-1}-(p-1)\)}\)\\ &\qquad \hbox{if}\  r =1,\dots,k-1\ \hbox{and}\ j>r,\\
& \\
  & O\(\de_1^{2-2p}\( {\de_{k-1} \over \de_k }\)^{{\al_j\over2}+(p-1) }\)
  =O\( \la^{-2^{k-1}(p-1)} \la ^{{3\over  2^k}\(2^{j-1}+(p-1)\)}   \)\\
  &\qquad \hbox{if}\  r=2,\dots,k \ \hbox{and}\  j<r.\\
  \end{aligned}  \right.
  \nonumber\\
  & \quad\hbox{(we compare with \eqref{er1.3})}\nonumber\\ &
=   o\(\la^{ {1 \over2^k} {(2-p) }}\),
  \end{align}
for some provided $p$ is close enough to 1.

\medskip
Let us estimate the third term in \eqref{er1.2}. For any $j=1,\dots,k$ and $r\not=j$,
\begin{align}\label{er5}
&  \int\limits_{A_r}\left|\la e^{ Pw_j(x)-{1\over2}\sum\limits_{\ell=1\atop \ell\not= j}^k P{w_\ell(x)}}\right|^pdx
\quad \hbox{(we apply \eqref{pro-exp})}\nonumber\\
& = O\(  \la^{ p}  \int\limits_{A_r}\left| {1\over(\de_j^{\al_j}+|x|^{\al_j})^2}\prod \limits_{\ell\not=j}^k {(\de_\ell^{\al_\ell}+|x|^{\al_\ell})}\right|^pdx\) \end{align}

Firstly, we consider the case  $k=2.$ We have only to estimate
\begin{align}\label{er5.1}
&       \int\limits_{A_1 }\left|  {\de_1^{\al_1}+|x|^{\al_1}\over(\de_2^{\al_2}+|x|^{\al_2})^2} \right|^pdx
\quad \hbox{and}\
      \int\limits_{A_2}\left|  {\de_2^{\al_2}+|x|^{\al_2}\over(\de_1^{\al_1}+|x|^{\al_1})^2} \right|^pdx  \end{align}
with  $\al_1=2,$ $\al_2=4 $ and $\de_1\sim\la,$ $\de_2\sim\la^{1\over4}.$
Therefore, we have
\begin{align}\label{er5.11}
&       \int\limits_{A_1}\left|  {\de_1^{\al_1}+|x|^{\al_1}\over(\de_2^{\al_2}+|x|^{\al_2})^2} \right|^pdx
=\int\limits_{\{x\in\Omega\ :\ |x|\le\sqrt{\de_1\de_2} \}}\left|  {\de_1^2+|x|^2\over(\de_2^4+|x|^4)^2} \right|^pdx\nonumber\\
&=O\(\int\limits_{\{  |y|\le\sqrt{\de_1\over\de_2} \}}\de_2^{2-8p}\de_1^{2p} {1  \over(1+|y|^4)^2p}dy \)+
O\(\int\limits_{\{ |y|\le\sqrt{\de_1\over\de_2} \}}\de_2^{2-6p} {  |y|^{2p} \over(1+|y|^4)^2}  dy\) \nonumber\\
&=O\(\de_2^{1-8p}\de_1^{2p+1}\)+O\(\de_2^{1-7p}\de_1^{p+1}\)=O\(\la^{5\over4}\) +O\(\la^{{5\over4}-{3\over4}p}\) \end{align}
and
\begin{align}\label{er5.12}
&       \int\limits_{A_2}\left|  {\de_2^{\al_2}+|x|^{\al_2}\over(\de_1^{\al_1}+|x|^{\al_1})^2} \right|^pdx = \int\limits_{\{x\in\Omega\ :\ |x|\ge\sqrt{\de_1\de_2}\}}\left|  {\de_2^4+|x|^4\over(\de_1^2+|x|^2)^2} \right|^pdx \nonumber\\
&=O\(\int\limits_{\{  |y|\ge\sqrt{\de_2\over\de_1} \}}\de_1^{2-4p}\de_2^{4p} {1  \over(1+|y|^2)^{2p}}dy \)+
O\(\int\limits_{\Omega}     dy\)\nonumber\\
&=O\(\de_2^{1+2p}\de_1^{1-2p}\)+O\(1\)=O\(\la^ {5\over4}\)+O\(1\)\end{align}
By \eqref{er5.11} and \eqref{er5.12}, we can compare   \eqref{er5}   with \eqref{er1.3} and we get
\begin{align}\label{er5.13}
&  \int\limits_{A_r}\left|\la e^{ Pw_r(x)-{1\over2}\sum\limits_{\ell=1\atop \ell\not= j}^k P{w_\ell(x)}}\right|^pdx
  = o\(\la^{ {1\over4}(2-p)}\)
  \end{align}
  provided $p$ is close enough to 1.

Now, let us consider the general case.
We estimate \eqref{er5} when $p=1.$ We have to estimate the following terms
\begin{align}\label{er5.14}
&  \la\ \de_1^{\al_1}\cdots\de_{j-1}^{\al_{j-1}}\de_{j+1}^{\al_{j+1}}\cdots\de_k ^{\al_k}\int\limits_{A_r}  {1\over(\de_j^{\al_j}+|x|^{\al_j})^2} dx =O(\la) \\\label{er5.15}
& \la \int\limits_{A_r}  {|x|^{\al_1+\cdots+\al_{j-1}+\al_{j+1}+\cdots+\al_k}\over(\de_j^{\al_j}+|x|^{\al_j})^2} dx = O\( \la^{{3\over2^k} (2^{k-1}-1)}\)\\\label{er5.16}
& \la\ \de_{\sigma_ {h+1}}^{\al_{\sigma_ {h+1}}}\cdots\de_{ \sigma_ {k-1}}^{\al_{\sigma_ {k-1}}} \int\limits_{A_r}  {|x|^{\al_{\sigma_1}+\cdots+ \al_{\sigma_ {h}}}\over(\de_j^{\al_j}+|x|^{\al_j})^2} dx=o(\la)
\end{align}
where $\{\sigma_1,\dots,\sigma_{k-1}\}$ is a permutation of the indices $\{1,\dots,j-1,j+1,\dots,k\}.$

Let us estimate \eqref{er5.14}.
\begin{align*}
&  \de_1^{\al_1}\cdots\de_{j-1}^{\al_{j-1}}\de_{j+1}^{\al_{j+1}}\cdots\de_k ^{\al_k}\int\limits_{A_r}  {1\over(\de_j^{\al_j}+|x|^{\al_j})^2} dx \quad  \hbox{(we scale $x=\de_jy$)} \nonumber\\
& = \de_1^{\al_1}\cdots\de_{j-1}^{\al_{j-1}}\de_{j+1}^{\al_{j+1}}\cdots\de_k ^{\al_k}\de_j^{2-2\al_j}\int\limits_{\left\{{\sqrt{\de_{r-1}\de_{r}}\over\de_j}\le|y|\le{\sqrt{\de_r\de_{r+1}}\over\de_j}\right\}}  {1\over(1+|y|^{\al_j})^2} dy  \nonumber\\
&=\left\{\begin{aligned}
&O\(\de_1^{\al_1}\cdots\de_{j-1}^{\al_{j-1}}\de_{j }^{2-\al_{j }}\la\)\ \hbox{if $j<k$ (because of \eqref{delta3})}\\
&O\({\de_1^{\al_1} }\de_2^{\al_2}\cdots\de_{j-1}^{\al_{j-1}}\de_{j+1}^{\al_{j+1}}\cdots\de_{k-1} ^{\al_{k-1}}\de_k^{2-\al_k }\)\ \hbox{if $j=k$ }\\
\end{aligned}\right.
\nonumber\\&=\left\{\begin{aligned}
&O\({\de_1^{\al_1}\over\de_j^{\al_j}}\cdots\de_{j-1}^{\al_{j-1}}\de_{j }^{2 }\la\)\ \hbox{if $j<k$  }\\
&O\({\de_1^{\al_1}\over\de_k^{\al_k}} \cdots\de_{j-1}^{\al_{j-1}}\de_{j+1}^{\al_{j+1}}\cdots\de_{k-1} ^{\al_{k-1}}\de_k^{2 }\)\ \hbox{if $j=k$ }\\
\end{aligned}\right.
\nonumber\\
&=O(1).\end{align*}
Therefore   \eqref{er5.14} follows.

Let us estimate \eqref{er5.15}. We immediately get
\begin{align}\label{er5.151}
 &  \int\limits_{A_r}  {|x|^{\al_1+\cdots+\al_{j-1}+\al_{j+1}+\cdots+\al_k}\over(\de_j^{\al_j}+|x|^{\al_j})^2} dx =O(1) \quad \hbox{if $j<k$},\end{align}
because the function is integrable at the origin, since
\begin{equation}\label{er5.100}
\al_1+\cdots+ \al_k=\sum\limits_{i=1}^k2^i=2(2^k-1)
\end{equation}
which implies
$$\al_1+\cdots+\al_{j-1}+\al_{j+1}+\cdots+\al_k+2-2\al_j= 2^j(2^{k+1-j}-3) >0.$$
If $j=k$ we scale $x=\de_ky$ and we get
\begin{align}\label{er5.152}
 &  \int\limits_{A_r}  {|x|^{\al_1+\cdots+\al_{j-1}+\al_{j+1}+\cdots+\al_k}\over(\de_j^{\al_j}+|x|^{\al_j})^2} dx \nonumber\\
 &=O\(\de_k^{\al_1+\cdots +\al_{k-1}+2-2\al_k}\int\limits_{\left\{{\sqrt{\de_{r-1}\de_{r}}\over\de_k}\le|y|\le{\sqrt{\de_r\de_{r+1}}\over\de_k}\right\}}  {|y|^{\al_1+\cdots +\al_{k-1}}\over(1+|y|^{\al_k})^2} dy\)\nonumber\\
 &=\(\de_k^{-2^k}\({\de_{k-1}\over\de_k}\)^{\al_1+\cdots +\al_{k-1}}\)\nonumber\\
 &=\({1\over\la}\la^{{3\over2^k} (2^{k-1}-1)}\), \hbox{because of \eqref{delta} and \eqref{delta0}}.\end{align}
By \eqref{er5.151} and \eqref{er5.152} we deduce \eqref{er5.15}

Let us estimate \eqref{er5.16}. It is useful to point out that $2\al_j=\al_{j+1}.$ Therefore, it is clear that
\begin{align}\label{er5.161}
 &\de_{\sigma_ {h+1}}^{\al_{\sigma_ {h+1}}}\cdots\de_{ \sigma_ {k-1}}^{\al_{\sigma_ {k-1}}} \int\limits_{A_r}  {|x|^{\al_{\sigma_1}+\cdots+ \al_{\sigma_ {h}}}\over(\de_j^{\al_j}+|x|^{\al_j})^2} dx=O(1)\\ &  \hbox{if $\sigma_i\ge j+1$ for some   $i=1,\dots,h $ or $\al_{\sigma_1}+\cdots+ \al_{\sigma_ {h}} +2-2\al_j\ge0.$}\nonumber
 \end{align}
It remains to consider the case $\sigma_i\le j-1$ for any   $i=1,\dots,h $ and $\al_{\sigma_1}+\cdots+ \al_{\sigma_ {h}} +2-2\al_j<0.$

In particular, it means that the set of indices $\{\sigma_{h+1},\dots,\sigma_{k-1}\}$ has to contain  a permutation of the indices $\{ j+1,\dots,k\}.$
Then,   we can write \eqref{er5.16} as
\begin{align}\label{er5.162}
 &\de_{j+1}^{\al_{j+1}}\cdots\de_k^{\al_k}\de_{\sigma_ {h+1}}^{\al_{\sigma_ {h+1}}}\cdots\de_{ \sigma_ {j-1}}^{\al_{\sigma_ {j-1}}} \int\limits_{A_r}  {|x|^{\al_{\sigma_1}+\cdots+ \al_{\sigma_ {h}}}\over(\de_j^{\al_j}+|x|^{\al_j})^2} dx\quad  \hbox{(we scale $x=\de_jy$)}   \nonumber\\
 &=\de_{j+1}^{\al_{j+1}}\cdots\de_k^{\al_k}\de_{\sigma_ {h+1}}^{\al_{\sigma_ {h+1}}}\cdots\de_{ \sigma_ {j-1}}^{\al_{\sigma_ {j-1}}}
  \de_j^{2-2\al_j+\al_{\sigma_ { 1}}+\cdots+\al_{\sigma_ { h}}} \int\limits_{\left\{{\sqrt{\de_{r-1}\de_{r}}\over\de_j}\le|y|\le{\sqrt{\de_r\de_{r+1}}\over\de_j}\right\}}   {|y|^{\al_{\sigma_1}+\cdots+ \al_{\sigma_ {h}}}\over(1+|y|^{\al_j})^2} dy\nonumber\\
 &= \de_{\sigma_ {h+1}}^{\al_{\sigma_ {h+1}}}\cdots\de_{ \sigma_ {j-1}}^{\al_{\sigma_ {j-1}}}
  \de_j^{2- \al_j+\al_{\sigma_ { 1}}+\cdots+\al_{\sigma_ { h}}} \int\limits_{\left\{{\sqrt{\de_{r-1}\de_{r}}\over\de_j}\le|y|\le{\sqrt{\de_r\de_{r+1}}\over\de_j}\right\}}   {|y|^{\al_{\sigma_1}+\cdots+ \al_{\sigma_ {h}}}\over(1+|y|^{\al_j})^2} dy\quad  \hbox{(because of \eqref{delta3})}\nonumber\\
  &=O\(
  \de_j^{2- \al_j+\al_{\sigma_ { 1}}+\cdots+\al_{\sigma_ { j-1}}} \int\limits_{\left\{{\sqrt{\de_{r-1}\de_{r}}\over\de_j}\le|y|\le{\sqrt{\de_r\de_{r+1}}\over\de_j}\right\}}   {|y|^{\al_{\sigma_1}+\cdots+ \al_{\sigma_ {h}}}\over(1+|y|^{\al_j})^2} dy\)\quad  \hbox{(because $\de_{\sigma_i}\le\de_j$)}\nonumber\\
  &=
  O\(\int\limits_{\left\{{\sqrt{\de_{r-1}\de_{r}}\over\de_j}\le|y|\le{\sqrt{\de_r\de_{r+1}}\over\de_j}\right\}}   {|y|^{\al_{\sigma_1}+\cdots+ \al_{\sigma_ {h}}}\over(1+|y|^{\al_j})^2} dy\)\nonumber\\
  &\quad  \hbox{(because $ \al_{\sigma_ { 1}}+\cdots+\al_{\sigma_ { j-1}}=\al_1+\dots+\al_{j-1}=2(2^{j-1}-1)=\al_j-2$)}\nonumber\\
 &=\left\{\begin{aligned}
&O\(\({\de_{j-1} \over\de_j } \)^{{\al_{\sigma_ { 1}}+\cdots+\al_{\sigma_ { h}}}+2\over2}\)\ \hbox{if $r<j$  }\\
&O\(\({\de_j \over\de_{j+1} } \)^{-(\al_{\sigma_ { 1}}+\cdots+\al_{\sigma_ { h}})-2+2\al_j\over2}\)\ \hbox{if $r>j$  }\\
\end{aligned}\right.
\nonumber\\
&=o(1).\end{align}
By \eqref{er5.161} and \eqref{er5.162} we deduce \eqref{er5.16}.

It is clear that if c$p$ is close enough to 1, by \eqref{er5}, \eqref{er5.14}, \eqref{er5.15} and \eqref{er5.16}
we can compare \eqref{er5} with \eqref{er1.3} and we  get
\begin{align}\label{er5.0}
&  \int\limits_{A_r}\left|\la e^{ Pw_j(x)-{1\over2}\sum\limits_{\ell=1\atop \ell\not= j}^k P{w_\ell(x)}}\right|^pdx
 = o\(   \la^{ {1\over2^k}(2-p)} \)\quad \hbox{if}\ r\not=j. \end{align}

 \medskip
 Finally, \eqref{er1} follows by \eqref{er1.2}, \eqref{er1.3}, \eqref{er6} and \eqref{er5.0}.
 That concludes the proof.
 \end{proof}

\section{The linear theory}\label{tre}

Let us consider the linear operator
\begin{align}\label{lla}&\mathcal{L}_{\la}(\boldsymbol\phi ):=\(L^1_\la(\boldsymbol\phi),\dots,L^k_\la(\boldsymbol\phi)\),\ \hbox{where}
\nonumber\\ &
L^i_\la(\boldsymbol\phi):=-\Delta \phi^i  - |x|^{\al_i-2}e^{w_i(x)}\phi^i   +{1\over2}\sum\limits_{j=1\atop j\not=i} |x|^{\al_j-2}e^{w_j(x)}\phi^j , \ j=1,\dots,k.
\end{align}

Let us   study  the invertibility of  the linearized operator $\mathcal{L}_{\la}.$

\begin{prop}\label{inv}
For any $p>1$ there exists $\la_0>0$ and $c>0$ such that for any $\la \in(0, \la_0)$ and for any $\boldsymbol\psi\in \(\mathrm{L}^{p}(\Omega)\)^k$ there exists a unique
$\boldsymbol\phi\in \(\mathrm{W}^{2, 2}(\Omega)\)^k\cap  \mathbf{H}_k$ solution of
$$ \mathcal{L}_{\la}(\boldsymbol\phi )=\boldsymbol\psi\ \hbox{in}\ \Omega,\ \boldsymbol\phi=0\ \hbox{on}\ \partial\Omega,
$$
which satisfies $$\|\boldsymbol\phi\| \leq c |\ln\la|  \|\boldsymbol\psi\|_{p}.$$
\end{prop}
\begin{proof}

We argue by contradiction. Assume there exist $p>1,$ sequences $\la_n\to0,$ $\boldsymbol\psi_n:=\(\psi^1_n,\dots,\psi^k_n\)\in \(\mathrm{L}^{\infty}(\Omega)\)^k$ and $\boldsymbol\phi_n:=\(\phi^1_n,\dots,\phi^k_n\)\in  \(\mathrm{W}^{2, 2}(\Omega)\)^k\cap  \mathbf{H}_k$
such that for any $i=1,\dots,k$
\begin{equation}\label{inv1}
-\Delta \phi^i _ n- |x|^{\al_i-2}e^{w_i(x)}\phi^i _ n +{1\over2}\sum\limits_{j=1\atop j\not=i} |x|^{\al_j-2}e^{w_j(x)}\phi^j_ n=\psi^i_n,\ \hbox{in}\ \Omega,\ \phi^i_n=0\ \hbox{on}\ \partial\Omega,
\end{equation}
with ${\de_1}_n ,\dots,{\de_k}_n  $   defined as in  \eqref{delta} and
\begin{equation}\label{inv2}
\|\boldsymbol\phi_n\| =\sum\limits_{i=1}^k \| \phi^i_n\|=1\quad\hbox{and}\quad    |\ln\la_n|  \|\boldsymbol\psi_n\|_p= |\ln\la_n|  \sum\limits_{i=1}^k \| \psi^i_n\|_p\to0.\end{equation}

For sake of simplicity, in the following we will omit the index $n$ in all the sequences and we write $\phi_i=\phi^i_n,$ $\psi_i=\psi^i_n$.
For any $i=1,\dots,k$ we define
 $ \tilde \phi_i(y):=\phi_i\({\de_i} y\) $ with $y\in \Omega_i:={\Omega\over \de_i}.$

{\em Step 1: we will show that
\begin{equation}\label{step1}  \tilde \phi_i (y)\to \gamma_i {1-|y|^{\al_i}\over1+|y|^{\al_i}}\ \hbox{  for some $ \gamma_i\in\rr $} \end{equation}
weakly in $\mathrm{H}_{\al_i}(\rr^2)$ and strongly in $\mathrm{L}_{\al_i}(\rr^2) $ (see \eqref{hjs} and \eqref{hjs}).}

First of all we claim  that
\begin{equation}\label{ut0}
\int\limits_\Omega |x|^{\al_i-2}e^{w_i(x)} \phi_i^2(x)dx=O(1)\ \hbox{for any $i=1,\dots,k.$}
\end{equation}

Indeed, we write    \eqref{inv1} for two functions $\phi_i$ and $\phi_\ell$ with $i\not=\ell $
\begin{align}\label{ut1}
&-\Delta \phi_i  - |x|^{\al_i-2}e^{w_i(x)}\phi_i   +{1\over2}\sum\limits_{j=1\atop j\not=i} |x|^{\al_j-2}e^{w_j(x)}\phi_j =\psi_i ,\ \hbox{in}\ \Omega,\ \phi_i =0\ \hbox{on}\ \partial\Omega,\\
&-\Delta \phi_\ell  - |x|^{\al_\ell-2}e^{w_\ell(x)}\phi_\ell   +{1\over2}\sum\limits_{j=1\atop j\not=\ell} |x|^{\al_j-2}e^{w_j(x)}\phi_j =\psi_\ell ,\ \hbox{in}\ \Omega,\ \phi^\ell=0\ \hbox{on}\ \partial\Omega,\\
\end{align}
then we subtract the two equations
\begin{align*}
&-\Delta \(\phi_i-\phi_\ell \)  - {3\over 2}\(|x|^{\al_i-2}e^{w_i(x)}\phi_i-   |x|^{\al_\ell-2}e^{w_\ell(x)}\phi_\ell\) = \psi_i-\psi_\ell ,\ \hbox{in}\ \Omega,\ \phi_i-\phi_\ell=0\ \hbox{on}\ \partial\Omega,\\
\end{align*}
  we multiply by $\phi_i ,$    we use \eqref{inv2} and we deduce
\begin{align*}
&    \int\limits_\Omega |x|^{\al_i-2}e^{w_i(x)} \phi_i^2(x)dx=  \int\limits_\Omega |x|^{\al_\ell-2}e^{w_\ell(x)} \phi_\ell (x)\phi_i (x)dx+O(1),\end{align*}
which implies (by summing over the index $\ell$)
\begin{align}\label{ut2}
& (k-1)   \int\limits_\Omega |x|^{\al_i-2}e^{w_i(x)} \phi_i^2(x)dx= \sum\limits_{\ell=1\atop\ell\not =i}^k \int\limits_\Omega |x|^{\al_\ell-2}e^{w_\ell(x)} \phi_\ell (x)\phi_i (x)dx+O(1). \end{align}
On the other hand, if we multiply the first equation \eqref{ut1} by $\phi_i $ and we use \eqref{inv2}, we get
\begin{align}\label{ut3}
&   \int\limits_\Omega |x|^{\al_i-2}e^{w_i(x)} \phi_i^2(x)dx= {1\over2}\sum\limits_{j=1\atop j\not =i}^k \int\limits_\Omega |x|^{\al_j-2}e^{w_j(x)} \phi_j(x)\phi_i (x)dx+O(1)\\\nonumber \\&={k-1\over2}   \int\limits_\Omega |x|^{\al_i-2}e^{w_i(x)} \phi_i^2(x)dx+O(1), \end{align}
where the last equality follows by \eqref{ut2}.
By \eqref{ut3} we immediately deduce \eqref{ut0} when $k\not=3.$
When $k=3$ we have to argue in a different way. For any index $i$, we write the   equation  \eqref{inv1} as
\begin{align}\label{ut4}
&-\Delta \phi_i   -{3\over2} |x|^{\al_i-2}e^{w_i(x)}\phi_i   +{1\over2} \sum\limits_{j=1}^3   |x|^{\al_j-j}e^{w_j(x)}\phi_j   =\psi_i ,\ \hbox{in}\ \Omega,\ \phi_i  =0\ \hbox{on}\ \partial\Omega ,
\end{align}
and we sum over the index $i=1,2,3$, so we get
$$-\Delta\(\phi_1+\phi_2+\phi_3\)=\psi_1+\psi_2+\psi_3 ,\ \hbox{in}\ \Omega,\ \phi_1+\phi_2+\phi_3 =0\ \hbox{on}\ \partial\Omega.$$
Since $\|\psi_1+\psi_2+\psi_3\|_p=o(1) $ (because of \eqref{inv2}), the standard regularity implies that
$\|\phi_1+\phi_2+\phi_3\|_\infty=o(1).$ Now, we multiply   equation in \eqref{ut4} by $\phi_i$ and we sum over the index $i=1,2,3$, so we obtain (using \eqref{inv2})
\begin{align*}
 {3\over2} \sum\limits_{i=1}^3\int\limits_\Omega |x|^{\al_i-2}e^{w_i(x)}\phi_i ^2(x)dx&=1+{1\over2} \sum\limits_{j=1}^3\int\limits_\Omega   |x|^{\al_j-j}e^{w_j(x)}\phi_j(x)
(\phi_1+\phi_2+\phi_3)(x)dx\\ &\le 1+ {1\over2}\|\phi_1+\phi_2+\phi_3\|_\infty \sum\limits_{j=1}^3\int\limits_\Omega   |x|^{\al_j-j}e^{w_j(x)}|\phi_j(x)|
 dx\\ &\le 1+ {1\over2}\|\phi_1+\phi_2+\phi_3\|_\infty \sum\limits_{j=1}^3\(\int\limits_\Omega   |x|^{\al_j-j}e^{w_j(x)}|\phi_j(x)|^2\)^{1/2},
 dx\end{align*}
 which implies
$$\sum\limits_{i=1}^3\int\limits_\Omega |x|^{\al_i-2}e^{w_i(x)}\phi_i ^2(x)dx=O(1),$$
because $\|\phi_1+\phi_2+\phi_3\|_\infty=o(1).$ That proves \eqref{ut0} when $k=3.$

Now, by \eqref{ut0} we deduce that each $\phi_i$ is bounded in the space $\mathrm{H}_{\al_i}(\rr^2)$
defined in \eqref{hjs}. Indeed, if we scale   we get
$$\int\limits_{\Omega_j}|\nabla\tilde\phi _j(y)|^2dy=\de^2_j\int\limits_{\Omega_j}|\nabla\tilde\phi_j (\de_j y)|^2dy=\int\limits_{\Omega }|\nabla\phi_j (x)|^2dx=1 $$
and
 $$\int\limits_{\Omega^j}2\al_j^2 {|y|^{\al_j-2}\over(1+|y|^{\al_j})^2}\(\tilde \phi_j(y)\)^2dy=\int\limits_{\Omega }2\al_j^2 {\de_j^{\al_j}|x|^{\al_j-2}\over(\de_j^{\al_j}+|x|^{\al_j})^2} \phi_j^2 (x)dx .$$

Finally,  by Proposition \eqref{compact} we can assume that (up to a subsequence)
  $\tilde\phi_j \rightharpoonup \tilde\phi^*_j $ weakly in $\mathrm{H}_{\al_j}(\rr^2)$ and strongly in $\mathrm{L}_{\al_j}(\rr^2).$

Now, by \eqref{inv1} we deduce that
each function $ \tilde \phi_j $ solves the problem
\begin{equation}\label{s2.3}
 -\Delta    \tilde\phi_j  =2\alpha_j^2{|y|^{\alpha_j-2}\over (1+|y|^{\alpha_j})^2}   \tilde\phi_j +\rho_j(y)\ \hbox{in}\ \Omega^j ,\  \tilde\phi_j =0\ \hbox{on}\ \partial\Omega_j ,
\end{equation}
where
$$\rho_j(y):={1\over2}\sum\limits_{i=1\atop i\not=j}^k2\alpha_i^2{ {\de_i}  ^{\alpha_i}{\de_j}  ^{\alpha_j}|y|^{\alpha_i-2}\over ({\de_i}  ^{\alpha_i}+{\de_j}  ^{\alpha_i}|y|^{\alpha_i})^2} \phi_i(\de_j y)+\de_j^2\psi_j(\de_j y).$$

Now, let $\varphi\in C^\infty_0(\rr^2)$ be a given function and let $\mathcal{K}$ its support. It is clear that if $n$ is large enough
$$\mathcal{K}\subset {A_j\over\de_j}=\left\{y\in \Omega^j\ :\ \sqrt{\de_{j-1}\over\de_j}\le |y|\le \sqrt{\de_{j+1}\over\de_j}\right\},$$
where $A_j$ is the annulus defined in \eqref{anelli}.
We multiply equation \eqref{s2.3} by $\varphi$ and we get
\begin{align*}
&\int\limits_{\mathcal{K}}\nabla  \tilde\phi_j   (y)\nabla\varphi(y) dy-\int\limits_{\mathcal{K}}2\alpha_j^2{|y|^{\alpha_j-2}\over (1+|y|^{\alpha_j})^2}    \tilde\phi_j  (y)\varphi(y)  dy=\int\limits_{\mathcal{K}} \rho_j(y)\varphi(y)dy.\end{align*}
Therefore, passing to the limit we get
\begin{align}\label{phi0}
&\int\limits_{\mathcal{K}}\nabla\tilde \phi_j^*(y)\nabla\varphi(y) dy-\int\limits_{\mathcal{K}}2\alpha_j^2{|y|^{\alpha_j-2}\over (1+|y|^{\alpha_j})^2} \tilde \phi_j^*(y)\varphi(y)  dy=0\ \forall\ \varphi\in C^\infty_0(\rr^2),\end{align}
because
\begin{align*}
&
\sum\limits_{i=1\atop i\not=j}^k\int\limits_{\mathcal{K}} 2\alpha_i^2{ {\de_i}  ^{\alpha_i}{\de_j}  ^{\alpha_j}|y|^{\alpha_i-2}\over ({\de_i}  ^{\alpha_i}+{\de_j}  ^{\alpha_j}|y|^{\alpha_i})^2}  \phi_i(\de_j y)\varphi(y)  dy\\ &=O\(\sum\limits_{i=1\atop i\not=j}^k\int\limits_{{A_j\over\de_j}} 2\alpha_i^2{ {\de_i}  ^{\alpha_i}{\de_j}  ^{\alpha_j}|y|^{\alpha_i-2}\over ({\de_i}  ^{\alpha_i}+{\de_j}  ^{\alpha_j}|y|^{\alpha_i})^2} | \phi_i(\de_j y)|  dy\)\ \hbox{(because $\mathcal{K}\subset {A_j\over\de_j}$)}\\
&=O\(\sum\limits_{i=1\atop i\not=j}^k \int\limits_{{A_j }} 2\alpha_i^2{ {\de_i}  ^{\alpha_i} |x|^{\alpha_i-2}\over ({\de_i}  ^{\alpha_i}+ |x|^{\alpha_i})^2} | \phi_i(x) |dx\)\ \hbox{(we scale $ x=\de_j y$)}\\
&=O\(\sum\limits_{i=1\atop i\not=j}^k\(\int\limits_{{A_j }}\left| 2\alpha_i^2{ {\de_i}  ^{\alpha_i} |x|^{\alpha_i-2}\over ({\de_i}  ^{\alpha_i}+ |x|^{\alpha_i})^2} dx \right|^p\)^{1/p} \| \phi_i \|_q\)\ \hbox{(we use H\"older's estimate)}\\ &=o(1) \ \hbox{(we use   estimate \eqref{er6} and the fact that  $| \phi_i |_q\le 1$)}   \end{align*}
and
$$\int\limits_{\mathcal{K}}\de_j^2\psi_j(\de_j y)\varphi(y)  dy=O\(\int\limits_{\Omega^j}\de_j^2|\psi_j(\de_j y)|  dy\)=O\(\int\limits_{\Omega } |\psi_j(x)|  dx\)=O(\|\psi_j\|_p)=o(1).$$
By \eqref{phi0} we deduce that $\tilde \phi_j^* $ is a solution to the equation
$$-\Delta \tilde \phi_j^* =2\alpha_j^2{|y|^{\alpha_j-2}\over (1+|y|^{\alpha_j})^2}\tilde \phi_j^*\ \hbox{in}\ \rr^2\setminus\{0\}.$$
Finally, since $\int\limits_{\rr^2}|\nabla \phi^j_0(y)|^2dy\le1$ it is standard to see that $\tilde \phi_j^*$ is a solution in the whole space $\rr^2.$
By Theorem \ref{esposito}  we get the claim.

\medskip
{\em Step 2: we will show that $\gamma_j=0$ for any $j=1,\dots,k.$}

Here we are  inspired by some ideas used by Gladiali-Grossi \cite{gg}.

We set
\begin{equation}\label{sigmai}\sigma_i(\lambda):= \ln\la \int\limits_{\Omega_i} 2\al_i^2{ |y|^{\al_i-2}\over \(1+|y|^{\al_ i}\)^2}\tilde \phi_i(y)dy.\end{equation}
We will show that
\begin{equation}\label{sigma}
\sigma_i :=\lim\limits_{\la\to0} \sigma_i(\lambda)=0\ \hbox{for any}\ i=1,\dots,k.
\end{equation}
We know that $\phi_i$ solves the problem (see \eqref{s2.3})
\begin{equation}\label{cru0}
-\Delta  \phi_i=2\al_i^2{\de_i^{\al_i}|x|^{\al_i-2}\over \(\de_i^{\al_i}+|x|^{\al_i}\)^2}\phi_i-{1\over2} \sum\limits_{j=1\atop j\not=i}^k2\al_j^2{\de_j^{\al_j}|x|^{\al_j-2}\over \(\de_j^{\al_j}+|x|^{\al_j}\)^2}\phi_j+\psi_i\ \hbox{in}\ \Omega,\ \phi_i=0\ \hbox{on}\ \partial\Omega.\end{equation}
Set $Z_i(x):={\de_i^{\al_i}-|x|^{\al_i}\over \de_i^{\al_i}+|x|^{\al_i}.}$ We know that
$Z_i$ solves (see Theorem \ref{esposito})
$$-\Delta Z_i=2\al_i^2{\de_i^{\al_i}|x|^{\al_i-2}\over \(\de_i^{\al_i}+|x|^{\al_i}\)^2}Z_i\quad\hbox{in}\ \rr^2.$$
Let $PZ_i$ be its projection onto $\mathrm{H}^1_0(\Omega) $ (see \eqref{pro}), i.e.
\begin{equation}\label{cru1}
-\Delta  PZ_i=2\al_i^2{\de_i^{\al_i}|x|^{\al_i-2}\over \(\de_i^{\al_i}+|x|^{\al_i}\)^2}Z_i\ \hbox{in}\ \Omega,\ PZ_i=0\ \hbox{on}\ \partial\Omega.\end{equation}
By maximum principle (see also Lemma \ref{pwi-lem}) we deduce that
\begin{equation}\label{pz}PZ_i(x)=Z_i(x)+1+O\(\de_i^{\al_i}\)={2\de_i^{\al_i} \over \de_i^{\al_i}+|x|^{\al_i} }+O\(\de_i^{\al_i}\)\end{equation}
frow which we get
\begin{equation}\label{cru2}
  PZ_i(\de_j y)=\left\{\begin{aligned}
  & O\({1\over|y|^{\al_i}}\({\de_i\over\de_j}\)^{\al_i}\) +O\(\de_i^{\al_i}\)\ \hbox{if}\ i<j,\\
&{2  \over 1+|y|^{\al_i} }+ O\(\de_i^{\al_i}\)\ \hbox{if}\ i=j,\\
& 2 +O\({ |y|^{\al_i}}\({\de_j\over\de_i}\)^{\al_i}\)+ O\(\de_i^{\al_i}\)\ \hbox{if}\ i>j.\\
 \end{aligned}\right.\end{equation}

Now,  we multiply \eqref{cru0} by $(\ln\la)  PZ_i$ and \eqref{cru1} by $(\ln\la)\phi.$ If we subtract the two equations obtained, we get
\begin{align*}
 & \ln\la\int\limits_\Omega   2\al_i^2{\de_i^{\al_i}|x|^{\al_i-2}\over \(\de_i^{\al_i}+|x|^{\al_i}\)^2}\phi_i(x)PZ_i(x)dx-{1\over2}
  \ln\la\sum\limits_{j=1\atop j\not=i }^k\int\limits_\Omega   2\al_j^2{\de_j^{\al_j}|x|^{\al_j-2}\over \(\de_j^{\al_j}+|x|^{\al_j}\)^2}\phi_j(x)PZ_i(x)dx
  \\ &+\ln\la\int\limits_\Omega \psi_i(x) PZ_i(x)dx =\ln\la \int\limits_\Omega   2\al_i^2{\de_i^{\al_i}|x|^{\al_i-2}\over \(\de_i^{\al_i}+|x|^{\al_i}\)^2}\phi_i(x) Z_i(x)dx    \end{align*}
and so
\begin{align}\label{cru3}
 &\ln\la\int\limits_\Omega   2\al_i^2{\de_j^{\al_i}|x|^{\al_i-2}\over \(\de_i^{\al_i}+|x|^{\al_i}\)^2}\phi_i(x) \(PZ_i(x)-Z_i(x)\) dx\nonumber\\ &-{1\over2}
 \ln\la\sum\limits_{j=1\atop j\not=i }^k\int\limits_\Omega   2\al_j^2{\de_j^{\al_j}|x|^{\al_j-2}\over \(\de_j^{\al_j}+|x|^{\al_j}\)^2}\phi_j(x)PZ_i(x)dx\nonumber\\ &+\ln\la\int\limits_\Omega \psi_i(x) PZ_i(x)dx=0.   \end{align}

We are going to pass to the limit in \eqref{cru3}.

The last term is
\begin{equation}\label{cru4}\ln\la\int\limits_\Omega \psi_i(x) PZ_i(x)dx=O\(|\ln\la|\|\psi_i\|_p\)=o(1),\end{equation}
because of \eqref{inv2} and since by \eqref{pz} we get
$\|PZ_i\|_\infty=O(1).$

The first term is
\begin{align}\label{cru5}
 &\ln\la\int\limits_\Omega   2\al_i^2{\de_i^{\al_i}|x|^{\al_i-2}\over \(\de_i^{\al_i}+|x|^{\al_i}\)^2}\phi_i(x) \(PZ_i(x)- Z_i(x)\)dx \nonumber\\ &\qquad \hbox{(we scale $x=\de_iy$ and we apply \eqref{pz})}\nonumber\\
&=\ln\la \int\limits_{\Omega_i} 2\al_i^2{ |y|^{\al_i-2}\over \(1+|y|^{\al_i}\)^2}\tilde \phi_i(y) dy +O\(\de_i^{\al_i}|\ln\la| \int\limits_{\Omega_i} 2\al_i^2{ |y|^{\al_i-2}\over \(1+|y|^{\al_i}\)^2}|\tilde\phi_i(y)| dy\)\nonumber\\ &\qquad  \hbox{(we use \eqref{sigmai} and \eqref{step1})}\nonumber\\
&= \sigma_i( \la) +o(1).\end{align}

We estimate the second term.
If $j\not=i$ we get
\begin{align}\label{cru6}
&\ln\la \int\limits_\Omega 2\al_j^2{\de_j^{\al_j}|x|^{\al_j-2}\over \(\de_j^{\al_j}+|x|^{\al_j}\)^2}\phi_j(x)PZ_i(x)dx\ \hbox{(we scale $x=\de_jy$)}\nonumber\\
&=\ln\la \int\limits_{\Omega_j} 2\al_j^2{ |y|^{\al_j-2}\over \(1+|y|^{\al_j}\)^2}\tilde\phi_j(y)PZ_i(\de_j y)dy\ \hbox{(we use \eqref{cru2})}\nonumber\\
&=\left\{\begin{aligned}
& 2\ln\la\int\limits_{\Omega_j} 2\al_j^2{ |y|^{\al_j-2}\over \(1+|y|^{\al_j}\)^2}\tilde\phi_j(y)dy   +\\
&\qquad +O\(|\ln\la|\int\limits_{\Omega_j} 2\al_j^2{ |y|^{\al_j-2}\over \(1+|y|^{\al_j}\)^2}|\tilde\phi_j(y)| \({ |y|^{\al_i}}\({\de_j\over\de_i}\)^{\al_i}+\de_i^{\al_i}\)dy\) \ & \hbox{if $j<i$}\\
&\\
&O\( |\ln\la|\int\limits_{\Omega_j} 2\al_j^2{ |y|^{\al_j-2}\over \(1+|y|^{\al_j}\)^2}|\tilde\phi_j(y)| \({1\over|y|^{\al_i}}\({\de_i\over\de_j}\)^{\al_i} +\de_i^{\al_i}\)dy\)\ & \hbox{if $j>i.$}\\
\end{aligned}\right.\nonumber\\
&\qquad\hbox{(we use \eqref{sigmai}, \eqref{crux}, \eqref{cruy} and  \eqref{cruz})}\nonumber\\
&=\left\{\begin{aligned}
&2\sigma_j(\la) +o(1)\ & \hbox{if $j<i$}\\
& o(1)\ & \hbox{if $j>i.$}\\
\end{aligned}\right.
\end{align}

 By \eqref{cru3}, \eqref{cru4}, \eqref{cru5} and  \eqref{cru6} we get
 $$\sigma_1(\la)=o(1)\ \hbox{and}\ \sigma_i(\la)- \sum\limits_{j=1}^{i-1}\sigma_j(\la)=o(1)\ \hbox{for any}\ i=2,\dots,k, $$
 which implies passing to the limit and using the definition of $\sigma_i$ given in  \eqref{sigma},
$$\sigma_1 =0\ \hbox{and}\ \sigma_i -\sum\limits_{j=1}^{i-1}\sigma_j =0\ \hbox{for any}\ i=2,\dots,k. $$
Therefore, \eqref{sigma} immediately follows.

We used the following three estimates.
If $j<i$ we have
\begin{align}\label{crux}
&\( |\ln\la|{\de_j\over\de_i}\)^{\al_i}\int\limits_{\Omega_j}{|y|^{\al_j+\al_i-2}\over\(1+|y|^{\al_j}\)^2}|\tilde\phi_j(y)|dy\ \hbox{(by H\"older's inequality)}\nonumber\\
&=O\( |\ln\la|\({\de_j\over\de_i}\)^{\al_i}\de_j^{2(1-p)\over p}\|\phi_j\|\(\int\limits_{\rr^2 }\({|y|^{\al_j+\al_i-2}\over\(1+|y|^{\al_j}\)^2} \)^pdy  \)^{1/p}\)\nonumber\\ &\qquad \hbox{(we use $\al_j >\al_i$ and we choose $p$ close to 1)}\nonumber\\
&=O\( |\ln\la|\({\de_j\over\de_i}\)^{\al_i}\de_j^{2(1-p)\over p} \)=o(1)
\end{align}
and if $j>i$ we have
\begin{align}\label{cruy}
&\( |\ln\la|{\de_i\over\de_j}\)^{\al_i}\int\limits_{\Omega_j}{1\over|y|^{\al_i-\al_j+2}\(1+|y|^{\al_j}\)^2}|\tilde\phi_j(y)|dy\ \hbox{(by H\"older's inequality)}\nonumber\\
&=O\( |\ln\la|\({\de_j\over\de_i}\)^{\al_i}\de_j^{2(1-p)\over p}\|\phi_j\|\(\int\limits_{\rr^2 }\({1\over|y|^{\al_i-\al_j+2}\(1+|y|^{\al_j}\)^2} \)^pdy  \)^{1/p}\)\nonumber\\ &\qquad \hbox{(we use $\al_i >\al_j$ and we choose $p$ close to 1)}\nonumber\\
&=O\( |\ln\la|\({\de_j\over\de_i}\)^{\al_i}\de_j^{2(1-p)\over p} \)=o(1);
\end{align}
moreover for any $i$ and $j$ we have
\begin{align}\label{cruz}
&  |\ln\la| \de_i ^{\al_i}\int\limits_{\Omega_j}{|y|^{\al_ j-2}\over\(1+|y|^{\al_j}\)^2}|\tilde\phi_j(y)|dy\ \hbox{(by H\"older's inequality)}\nonumber\\
&=O\( |\ln\la|\de_i ^{\al_i}\de_j^{2(1-p)\over p}\|\phi_j\|\(\int\limits_{\rr^2 }\({|y|^{\al_ j-2}\over \(1+|y|^{\al_j}\)^2} \)^p dy \)^{1/p}\)\nonumber\\ &\qquad \hbox{(  we choose $p$ close to 1)}\nonumber\\
&=O\( |\ln\la|\de_i ^{\al_i}\de_j^{2(1-p)\over p} \)=o(1).
\end{align}

 \medskip

Finally, we have all the ingredients to show that
\begin{equation}\label{ai}
\gamma_i  =0\ \hbox{for any}\ i=1,\dots,k.
\end{equation}

  We know that $Pw_i$  solves the problem
\begin{equation}\label{cr1}
-\Delta  Pw_i=2\al_i^2{\de_i^{\al_i}|x|^{\al_i-2}\over \(\de_i^{\al_i}+|x|^{\al_i}\)^2} \ \hbox{in}\ \Omega,\ Pw_i=0\ \hbox{on}\ \partial\Omega.\end{equation}
Now,  we multiply \eqref{cr1} by $ \phi_i$  and \eqref{cru0} by $ Pw_i,$   we subtract the two equations  and  we get
\begin{align}\label{cr2}
& \int\limits_\Omega   2\al_i^2{\de_i^{\al_i}|x|^{\al_i-2}\over \(\de_i^{\al_i}+|x|^{\al_i}\)^2}\phi_i(x)  dx
\nonumber \\ &= \int\limits_\Omega   2\al_i^2{\de_i^{\al_i}|x|^{\al_i-2}\over \(\de_i^{\al_i}+|x|^{\al_i}\)^2}\phi_i(x) Pw_i dx
 -{1\over2} \sum\limits_{j=1\atop j\not=i }^k\int\limits_\Omega   2\al_j^2{\de_j^{\al_j}|x|^{\al_j-2}\over \(\de_j^{\al_j}+|x|^{\al_j}\)^2}\phi_j(x)Pw_i(x)dx \nonumber \\ &+ \int\limits_\Omega \psi_i(x) Pw_i(x)dx.  \end{align}

We want to pass to the limit in \eqref{cr2}.

The L.H.S. of \eqref{cr2} reduces to
 \begin{align}\label{cr4}
& \int\limits_\Omega   2\al_i^2{\de_i^{\al_i}|x|^{\al_i-2}\over \(\de_i^{\al_i}+|x|^{\al_i}\)^2}\phi_i(x)  dx\ \hbox{(we scale $x=\de_iy$)}\nonumber\\
&= \int\limits_{\Omega_i} 2\al_i^2{ |y|^{\al_i-2}\over \(1+|y|^{\al_i}\)^2}\tilde\phi_i(y) dy =o(1)\  \hbox{(because of \eqref{ex1} and \eqref{step1}).}
  \end{align}

The last term of the R.H.S. of \eqref{cr2} gives
 \begin{align}\label{cr41}\int\limits_\Omega \psi_i(x) Pw_i(x)dx=O\(|\ln\la|\|\psi_i\|_p\)o(1),\end{align}
  because of \eqref{inv2} and since by   \eqref{pro-exp}   we get
$\|Pw_i\|_\infty= O(|\ln\la|).$

Finally, we claim that  the first term of the R.H.S. of \eqref{cr2} is
  \begin{align}\label{cr42}
& \int\limits_\Omega   2\al_i^2{\de_i^{\al_i}|x|^{\al_i-2}\over \(\de_i^{\al_i}+|x|^{\al_i}\)^2}\phi_i(x) Pw_i dx
 -{1\over2} \sum\limits_{j=1\atop j\not=i }^k\int\limits_\Omega   2\al_j^2{\de_j^{\al_j}|x|^{\al_j-2}\over \(\de_j^{\al_j}+|x|^{\al_j}\)^2}\phi_j(x)Pw_i(x)dx \nonumber \\ &=\left\{\begin{aligned}
  &4\pi\al_i\(\gamma_i-\sum\limits_{j=i+1}^k \gamma_j\)+o(1)&\ \hbox{if}\ i=1,\dots,k-1,\\
  & 4\pi\al_k\gamma_k +o(1)&\ \hbox{if}\ i=k.\\
  \end{aligned}\right.\end{align}
  Therefore, passing to the limit, by \eqref{cr2}, \eqref{cr4}, \eqref{cr41} and \eqref{cr42}
  we immediately get
  $$ \gamma_k =0\ \hbox{and}\ \gamma_i-\sum\limits_{j=i+1}^k \gamma_j=0\ \hbox{for any}\ i=1,\dots,k-1,$$
  which implies \eqref{ai}.

It only remains to    prove \eqref{cr42}. We have
\begin{align*}
&\int\limits_\Omega   2\al_j^2{\de_j^{\al_j}|x|^{\al_j-2}\over \(\de_j^{\al_j}+|x|^{\al_j}\)^2}\phi_j(x)Pw_i(x)dx\ \hbox{(we scale $x=\de_jy$)}\nonumber\\
&=\int\limits_{\Omega_j}   2\al_j^2{ |y|^{\al_j-2}\over \(1+|y|^{\al_j}\)^2}\tilde\phi_j(y)Pw_i(\de_jy)dy\ \hbox{(we use \eqref{pwi})}\nonumber\\ \end{align*}
\begin{align*}
&=\left\{\begin{aligned}
&  \int\limits_{\Omega_j} 2\al_j^2{ |y|^{\al_j-2}\over \(1+|y|^{\al_j}\)^2}\tilde\phi_j(y)\(-2\al_i \ln\de_i+h_i(0)\)dy   +\\
&\qquad +O\( \int\limits_{\Omega^j} 2\al_j^2{ |y|^{\al_j-2}\over \(1+|y|^{\al_j}\)^2}|\tilde\phi_j(y)|
 \({ |y|^{\al_i}}\({\de_j\over\de_i}\)^{\al_i}+\de_j|y|+\de_i^{\al_i}\)dy\)
\ & \hbox{if $j<i$}\\
&\\
 &  \int\limits_{\Omega_i} 2\al_i^2{ |y|^{\al_i-2}\over \(1+|y|^{\al_i}\)^2}\tilde\phi_i(y)\(-2\al_i\ln \de_i-2 \ln (1+|y|^{\al_i})   +h_i(0)\)dy   +\\
&\qquad +O\( \int\limits_{\Omega^i} 2\al_i^2{ |y|^{\al_i-2}\over \(1+|y|^{\al_i}\)^2}|\tilde\phi_i(y)|
 \( \de_i|y|+\de_i^{\al_i}\)dy\)
\ & \hbox{if $j=i$}\\
&\\
 &  \int\limits_{\Omega_j} 2\al_j^2{ |y|^{\al_j-2}\over \(1+|y|^{\al_j}\)^2}\tilde\phi_j(y)\(-2\al_i\ln\(\de_j|y|\) +h_i(0)\)dy   +\\
&\qquad +O\( \int\limits_{\Omega^j} 2\al_j^2{ |y|^{\al_j-2}\over \(1+|y|^{\al_j}\) }|\tilde\phi_j(y)|
 \({1\over|y|^{\al_i}}\({\de_i\over\de_j}\)^{\al_i}+\de_j|y|+\de_i^{\al_i}\)dy\)
\ & \hbox{if $j>i$}\\
&\\
\end{aligned}\right.\nonumber\\
& \hbox{(we use  the relation between $\de_i$ and $\la$ in \eqref{delta} and we use \eqref{crux}, \eqref{cruy}, \eqref{cruz}  and \eqref{cruw})}\\\end{align*}
\begin{align*}
 &=\left\{\begin{aligned}
&  \int\limits_{\Omega_j} 2\al_j^2{ |y|^{\al_j-2}\over \(1+|y|^{\al_j}\)^2}\tilde\phi_j(y) \[-2\al_i \ln d_i-2\(2(k-i)+1\)\ln\la+h_i(0)\] dy   \\
&\qquad +o(1) \  \hbox{if $j<i$}\\
&\\
 &  \int\limits_{\Omega_i} 2\al_i^2{ |y|^{\al_i-2}\over \(1+|y|^{\al_i}\)^2}\tilde\phi_i(y)\[-2\al_i \ln d_i-2\(2(k-i)+1\)\ln\la-2 \ln (1+|y|^{\al_i}) +h_i(0) \]dy  \\
&\qquad  +o(1) \  \hbox{if $j=i$}\\
&\\
 &  \int\limits_{\Omega_j} 2\al_j^2{ |y|^{\al_j-2}\over \(1+|y|^{\al_j}\)^2}\tilde\phi_j(y)\[-2\al_i \ln d_j-2\(2(k-j)+1\)\ln\la -2\al_i\ln|y| +h_i(0) \] dy  \\
&\qquad   + o(1) \  \hbox{if $j>i$} \\
\end{aligned}\right.\nonumber\\
& \hbox{(we use  the definition of $\sigma_i$ in   \eqref{sigmai} and  we use \eqref{step1} and  \eqref{ex1})}\\
 \end{align*}
 \begin{align*}
 &=\left\{\begin{aligned} &-2\(2(k-i)+1\)\sigma_j(\la) +o(1)
 \\
&\qquad\qquad \hbox{if $j<i$}\\
&\\
 & -2\(2(k-i)+1\)\sigma_i(\la)+ \int\limits_{\Omega_i} 2\al_i^2{ |y|^{\al_i-2}\over \(1+|y|^{\al_i}\)^2}\tilde\phi_i(y)\[ -2 \ln (1+|y|^{\al_i})   \]dy    +o(1)  \\
&\qquad \qquad  \hbox{if $j=i$}\\
&\\
 &  -2\(2(k-j)+1\)\sigma_j(\la)+\int\limits_{\Omega_j} 2\al_j^2{ |y|^{\al_j-2}\over \(1+|y|^{\al_j}\)^2}\tilde\phi_j(y)\[  -2\al_i\ln|y|   \] dy     + o(1)  \\
&\qquad \qquad \hbox{if $j>i$} \\
\end{aligned}\right.\nonumber\\
& \hbox{(we use     \eqref{sigma}  and  \eqref{step1}    because $\ln (1+|y|^{\al_j}),\ln|y|\in \mathrm{L}_{\al_j}(\rr^2)$ )}\\
 \end{align*}
   \begin{align*}
   &=\left\{\begin{aligned} &o(1)
 &\  \hbox{if $j<i$}\\
&\\
 &  \gamma_i \int\limits_{\rr^2} 2\al_i^2{ |y|^{\al_i-2}\over \(1+|y|^{\al_i}\)^2}{ 1-|y|^{\al_i }\over  1+|y|^{\al_i}  } \[ -2 \ln (1+|y|^{\al_i})   \]dy    +o(1)   &\  \hbox{if $j=i$}\\
&\\
 &    \gamma_j \int\limits_{\rr^2} 2\al_j^2{ |y|^{\al_j-2}\over \(1+|y|^{\al_j}\)^2}{ 1-|y|^{\al_j }\over  1+|y|^{\al_j}  }\[  -2\al_i\ln|y|   \] dy     + o(1)  &\ \hbox{if $j>i$} \\
\end{aligned}\right.\nonumber\\
& \hbox{(we use    \eqref{ex2} and \eqref{ex3})}\\
 \end{align*}
   \begin{align}\label{cr5} &=\left\{\begin{aligned}
&  o(1) \ & \hbox{if $j<i$}\\
&\\
 &  4\pi\al_i \gamma_i  +o(1) \ & \hbox{if $j=i$}\\
&\\
 &  8\pi\al_i  \gamma_j    + o(1) \ & \hbox{if $j>i$} \\
\end{aligned}\right.
 \end{align}

If we sum \eqref{cr5} over the index $j$ we get \eqref{cr42}.

We used the following estimate.
For any   $j$ we have
\begin{align}\label{cruw}
&    \de_j  \int\limits_{\Omega_j}{|y|^{\al_ j-1 }\over\(1+|y|^{\al_j}\)^2}|\tilde\phi_j(y)|dy\ \hbox{(by H\"older's inequality)}\nonumber\\
&=O\( \de_j  \de_j^{2(1-p)\over p}\|\phi_j\|\(\int\limits_{\rr^2 }\({|y|^{\al_ j-1}\over \(1+|y|^{\al_j}\)^2} \)^p dy \)^{1/p}\)\nonumber\\ &\qquad \hbox{(  we choose $p$ close to 1)}\nonumber\\
&=O\(   \de_j^{2-p\over p} \)=o(1).
\end{align}

  A straightforward computation leads to
  \begin{align}\label{ex1}
 &\int\limits_\Omega   2\al_i^2{ |y|^{\al_i-2}\over \(1+|y|^{\al_i}\)^2}{  1-|y|^{\al_i} \over 1+|y|^{\al_i} }dy=0,\\
\label{ex2}
 &\int\limits_\Omega   2\al_i^2{ |y|^{\al_i-2}\over \(1+|y|^{\al_i}\)^2}{  1-|y|^{\al_i} \over 1+|y|^{\al_i} }\ln\(1+|y|^{\al_i}\)^2dy=-4\pi\al_i,\\
\label{ex3}
 &\int\limits_\Omega   2\al_i^2{ |y|^{\al_i-2}\over \(1+|y|^{\al_i}\)^2}{  1-|y|^{\al_i} \over 1+|y|^{\al_i} }\ln|y|dy=-4\pi.
 \end{align}

\medskip
{\em Step 3: we will show that a contradiction arises!}

 We multiply each equation \eqref{inv1}  by $\phi_i$, we sum over the indices $i$'s and we get
\begin{align*}
1 &=\sum\limits_{i=1}^k\int\limits_{\Omega}2\alpha_i^2{ {\de_i}  ^{\alpha_i}|x|^{\alpha_i-2}\over ({\de_i}  ^{\alpha_i}+|x|^{\alpha_i})^2}	 \phi^2_i(x) dx
-{1\over2}\sum\limits_{i,j=1\atop i\not=j}^k\int\limits_{\Omega}2\alpha_j^2{ {\de_j}  ^{\alpha_j}|x|^{\alpha_j-2}\over ({\de_j}  ^{\alpha_j}+|x|^{\alpha_j})^2} \phi_j  (x) \phi_i (x)dx\\ &+\sum\limits_{i=1}^k\int\limits_{\Omega}\psi_i(x)\phi_i(x)dx    \\
&=\sum\limits_{i=1}^k \int\limits_{\Omega/\de_i} 2\alpha_i^2{  |y|^{\alpha_i-2}\over (1+|y|^{\alpha_i})^2} \tilde\phi_i^2(y) dy\\ &-{1\over2}\sum\limits_{i,j=1\atop i\not=j}^k\int\limits_{\Omega\setminus B(0,\rho)}2\alpha_j^2{ {\de_j}  ^{\alpha_j}|x|^{\alpha_j-2}\over ({\de_j}  ^{\alpha_j}+|x|^{\alpha_j})^2} \phi_j  (x) \phi_i (x)dx\\ &-{1\over2}\sum\limits_{i,j=1\atop i\not=j}^k\int\limits_{B(0,\rho/\de_j)}
2\alpha_j^2{  |y|^{\alpha_j-2}\over (1+|y|^{\alpha_j})^2} \tilde\phi_j(y)\phi_i(\de_j y)dy\\ & +\sum\limits_{i=1}^k\int\limits_{\Omega}\psi_i(x)\phi_i(x)dx   \\
  &=\sum\limits_{i=1}^k \int\limits_{\Omega/\de_i} 2\alpha_i^2{  |y|^{\alpha_i-2}\over (1+|y|^{\alpha_i})^2} \tilde\phi_i^2(y) dy+O\(\sum\limits_{i,j=1\atop i\not=j}^k{\de_j}  ^{\alpha_j}\|\phi_i\|\|\phi_j\|\)\\ &
  +O\(\sum\limits_{i,j=1\atop i\not=j}^k\( \int\limits_{B(0,\rho/\de_j)} 2\alpha_i^2{  |y|^{\alpha_i-2}\over (1+|y|^{\alpha_i})^2} \tilde\phi_i^2(y) dy\)^{1/2}\( \int\limits_{B(0,\rho/\de_j)} 2\alpha_i^2{  |y|^{\alpha_i-2}\over (1+|y|^{\alpha_i})^2} \phi_i^2(\de_j y) dy\)^{1/2}\)
  \\ &+O\(\sum\limits_{i=1}^k \|\psi_i\|_p\|\phi_i\|\) 
  \\ &=o(1)
  \end{align*}
 because     $\tilde \phi_i\to0$   strongly in $ \mathrm{L}_{\al_i}(\mathbb R^2) $ for any $i=1,\dots,k$  
so that
$$\int\limits_{\Omega/\de_i} 2\alpha_i^2{  |y|^{\alpha_i-2}\over (1+|y|^{\alpha_i})^2} \tilde\phi_i^2(y) dy=o(1)$$
and by \eqref{1} we deduce that
$$\int\limits_{B(0,\rho/\de_j)} 2\alpha_i^2{  |y|^{\alpha_i-2}\over (1+|y|^{\alpha_i})^2} \phi_i^2(\de_j y) dy=O(1).$$
Therefore a contradiction arises.

\end{proof}

In the next lemma we establish the decay of each component $\phi_i$ around the origin.
  
\begin{lemma}\label{lem1}
Let $\phi_i$ be the solution of equation \eqref{inv1}
  i.e.
  \begin{equation}\label{46}
  -\Delta\phi_i=|x|^{\alpha_i-2}e^{w_i(x)}\phi_i-\frac12\sum_{l\not=i}
  |x|^{\alpha_l-2}e^{w_l(x)}\phi_l +\psi_i\ \hbox{in}\ \Omega ,\ \phi_i=0\ \hbox{on}\ \partial\Omega.
  \end{equation}
 Let $\beta\in(0,1) $ and  $\rho>0$ be fixed small enough. There exists $\lambda_0>0$ and there exist positive constants $a,b,c$ such that for any $\lambda\in(0,\lambda_0)$
  \begin{equation}
  \label{1}
 | \phi_i(\delta_j x)|\le a{|\ln |x||\over|x|^{\beta/4}}+b|\ln |x||+c\ \hbox{for any}\ x\in  B(0,\rho/ \delta_j).
  \end{equation}
  
  \end{lemma}
  \begin{proof}
  Let $G(x,y)=-{1\over 2\pi} \ln|x-y|+H(x,y)$ be Green's function.
  Equation \eqref{46} can be rewritten as
   \begin{equation}\label{2}\begin{aligned}
   \phi_i({\rm x})&=\sum_{l=1}^k\int\limits_\Omega G({\rm x},y) c_lp_l(y)\phi_l (y)dy+\int\limits_\Omega G({\rm x},y)\psi_i(y) dy .
  \end{aligned}\end{equation} 
  where   $c_i=1$ and $c_l=-\frac12$ if $l\not=i,$ $p_l(y):=|y|^{\alpha_l-2}e^{w_l(y)}.$
  \\
  First of all, we prove that
  \begin{equation}\label{20}
  \left|\int\limits_\Omega G({\rm x},y)\psi_i(y) dy\right|\le c\ \hbox{for any}\ {\rm x}\in B(0,\rho).\end{equation}
  Indeed, we have
   $$\begin{aligned}
 \left|\int\limits_\Omega G({\rm x},y)\psi_i(y) dy\right|&\le c  \int\limits_{\Omega } |\ln|  {\rm x}-  y|| |\psi_i(y)|dy +  \int\limits_{\Omega } |H(  {\rm x}, y)|  |\psi_i(y)| dy\\
 &\le c\(\int\limits_{\Omega } |\ln|  {\rm x}-  y| |^{p\over p-1}dy \)^{p-1\over p}\|\psi_i\|_p+c\(\int\limits_{\Omega } |H({\rm x}, y)|   ^{p\over p-1}dy \)^{p-1\over p}\|\psi_i\|_p\\ &\le
 c\(\int\limits_{\{|  {\rm x}-  y|\le1\}} |\ln|  {\rm x}-  y| |^{p\over p-1}dy \)^{p-1\over p}\|\psi_i\|_p\\ &+c\(\int\limits_{\{2{\rm diam}\Omega\ge |  {\rm x}-  y|\ge1\}} |\ln|  {\rm x}-  y| |^{p\over p-1}dy \)^{p-1\over p}\|\psi_i\|_p\\ &+c\max\limits_{y\in\Omega\atop {\rm x}\in B(0,\rho)}|H({\rm x}, y)|  \({\rm meas}(\Omega) \)^{p-1\over p}\|\psi_i\|_p\\
 &\le c.
  \end{aligned}$$ 
  
 Let us set ${\rm x}=\delta_j x$. Next we prove that for any $l=1,\dots,k$
  \begin{equation}\label{30}
  \left| \int\limits_\Omega G(\delta_j x,y)  p_l(y)\phi_l (y)dy\right|\le a{|\ln |x||\over|x|^{\beta/4}}+b|\ln |x||+c\ \hbox{for any}\ { x}\in B(0,\rho/\delta_j),
  \end{equation}
  where $a,$ $b$ and $c$ does not depend on ${x}$.
  
  Let us scale $y=\delta_l z.$  Therefore we get
    \begin{equation}\label{3}\begin{aligned}
   \int\limits_\Omega G(\delta_j x,y)  p_l(y)\phi_l (y)dy&=  -{1\over2\pi}  \int\limits_{\Omega_l} \ln|\delta_j x-\delta_l z| \tilde p_l(z)\tilde \phi_l (z)dz \\ &+  \int\limits_{\Omega_l} H(\delta_j x,\delta_l z)  \tilde p_l(z)\tilde \phi_l (z)dz
  \end{aligned}\end{equation} 
 where we set $ \tilde p_l(z):={|z|^{\alpha_l-2}\over(1+|z|^{\alpha_l} )^2} $ and $\tilde \phi_l(z):=\phi_l(\delta_l z).$
 In the following $c,$ $c_1$ and $c_2$ will denote positive constants which do not depend on $x.$\\
 
 By mean value theorem we get
 \begin{equation}\label{4}\begin{aligned}
 \int\limits_{\Omega_l} H(\delta_j x,\delta_l z)  \tilde p_l(z)\tilde \phi_l (z)dz&=\int\limits_{\Omega_l} H(0,0) \tilde p_l(z)\tilde \phi_l (z)dz\\ &+O\(\int\limits_{\Omega_l} \(\delta_j |x|+\delta_l |z|\) \tilde p_l(z)|\tilde \phi_l (z)|dz\)\\
 &\le c,\end{aligned}
 \end{equation} 
 because
 $$H(0,0) \int\limits_{\Omega_l}  \tilde p_l(z)\tilde \phi_l (z)dz=o(1),$$
 $$ \int\limits_{\Omega_l}   \tilde p_l(z)|\tilde \phi_l (z)|dz\le 
 \(\int\limits_{\Omega_l}   \tilde p_l(z)\tilde \phi^2_l (z)dz\)^{\frac12}
 \(\int\limits_{\Omega_l}   \tilde p_l(z) dz\)^{\frac12}=o(1)$$
 and
 $$ \delta_l\int\limits_{\Omega_l}   \tilde p_l(z)|z| |\tilde \phi_l (z)|dz\le \delta_l
 \(\int\limits_{\Omega_l}   \tilde p_l(z)\tilde \phi^2_l (z) dz\)^{\frac12}
 \(\int\limits_{\Omega_l}   \tilde p_l(z) |z|^2dz\)^{\frac12}=o(1)$$
since
 \begin{equation}\label{5}
 \int\limits_{\Omega_l}   \tilde p_l(z)\tilde \phi^2_l (z)dz=o(1).
 \end{equation}
  Moreover, 
  \begin{equation}\label{6}\begin{aligned}
  & \int\limits_{\Omega_l} \ln|\delta_j x-\delta_l z| \tilde p_l(z)\tilde \phi_l (z)dz \\ &  = \ln\delta_j\int\limits_{\Omega_l}  \tilde p_l(z)\tilde \phi_l (z)dz  +\int\limits_{\Omega_l} \ln|  x-{\delta_l\over\delta_j} z| \tilde p_l(z)\tilde \phi_l (z)dz\\
   &=o(1)+ \int\limits_{\Omega_l} \ln|  x-{\delta_l\over\delta_j} z| \tilde p_l(z)\tilde \phi_l (z)dz \\ &\le
   c_1{|\ln|x||\over |x|^{\beta/4}}+c_2\ \hbox{if}\ l<j
  \end{aligned}\end{equation} 
 or 
 \begin{equation}\label{7}\begin{aligned}
 &  \int\limits_{\Omega_l} \ln|\delta_j x-\delta_l z| \tilde p_l(z)\tilde \phi_l (z)dz \\ &  = \ln\delta_l\int\limits_{\Omega_l}  \tilde p_l(z)\tilde \phi_l (z)dz +\int\limits_{\Omega_l} \ln|  {\delta_j\over\delta_l}x- z| \tilde p_l(z)\tilde \phi_l (z)dz\\
   & =o(1)+ \int\limits_{\Omega_l} \ln|  {\delta_j\over\delta_l}x- z| \tilde p_l(z)\tilde \phi_l (z)dz\\ &\le
   c_1{|\ln|x| |}+c_2\ \hbox{if}\ l\ge j.
  \end{aligned}\end{equation}
  Indeed we used the following facts. First of all, we know that
  $$\ln\lambda\int\limits_{\Omega_l}  \tilde p_l(z)\tilde \phi_l (z)dz =o(1).$$
  Moreover,
  if $l<j$ 
 \begin{equation}\label{8}\begin{aligned} &\left|\int\limits_{\Omega_l}  \ln|  x-{\delta_l\over\delta_j} z| \tilde p_l(z)\tilde \phi_l (z)dz\right|\\ &= \(\int\limits_{\Omega_l} \ln^2|  x-{\delta_l\over\delta_j} z| \tilde p_l(z)  dz\)^{\frac12}\(\int\limits_{\Omega_l}  \tilde p_l(z)\tilde \phi_l^2 (z)dzdz\)^{\frac12}\\ &
 \le c\(\int\limits_{\mathbb R^2} \ln^2|  x-{\delta_l\over\delta_j} z| \tilde p_l(z)  dz\)^{\frac12}\\
 &\le  \( \int\limits_{\mathbb R^2}  {\ln^4|  x-{\delta_l\over\delta_j} z|\over \(1+|z|^{2+\beta}\)}dz\)^{\frac14}\(\int\limits_{\mathbb R^2}  \(1+|z|^{2+\beta}\)\tilde p_l^2(z)  dz\)^{\frac14}\\
 &\le c {|\ln |x||\over|x|^{\beta/4}}\ \hbox{because of \eqref{s2}}
 \end{aligned}\end{equation}
 and if $l\ge j$
 \begin{equation}\label{9}\begin{aligned} &\left|\int\limits_{\Omega_l}  \ln|  {\delta_j\over\delta_l}x- z| \tilde p_l(z)\tilde \phi_l (z)dz\right|\\ &= \(\int\limits_{\Omega_l} \ln^2|  {\delta_j\over\delta_l}x- z| \tilde p_l(z)  dz\)^{\frac12}\(\int\limits_{\Omega_l}  \tilde p_l(z)\tilde \phi_l^2 (z)dzdz\)^{\frac12}
 \\ &\le c\(\int\limits_{\mathbb R^2} \ln^2|  {\delta_j\over\delta_l}x- z| \tilde p_l(z)  dz\)^{\frac12}\\ &\le \( \int\limits_{\mathbb R^2}  {\ln^4|  {\delta_j\over\delta_l}x- z|\over \(1+|z|^{2+\beta}\)}dz\)^{\frac14}\(\int\limits_{\mathbb R^2}  \(1+|z|^{2+\beta}\)\tilde p_l^2(z)  dz\)^{\frac14}\\
 &\le c_1|\ln |x||+c_2\ \hbox{because of \eqref{s1}.}
 \end{aligned}\end{equation}

  \end{proof}
  \begin{lemma} Let $\beta\in(0,1).$
  There exists $\lambda_0\ge1$ and positive constants $a,b,c$ such that for any $\lambda\in(0,\lambda_0)$
\begin{equation}\label{s1}\int\limits_{\mathbb R^2}  {\ln^4|  \lambda x- z|\over 1+|z|^{2+\beta}}dz\le a\ln^4|x|+b\quad\hbox{for any}\ x\in\mathbb R\end{equation}
  and
   \begin{equation}\label{s2}\int\limits_{\mathbb R^2}  {\ln^4|   x- \lambda z|\over 1+|z|^{2+\beta}}dz\le c{\ln^4|x|\over|x|^\beta}\quad\hbox{for any}\ x\in\mathbb R.\end{equation}
  \end{lemma}
  \begin{proof}
  In the following, $c$ will denote  a constant which does not depend on $x.$\\
  
  Let us prove \eqref{s1}.
  First of all, we have
 $$\begin{aligned}  \int\limits_{\mathbb R^2}  {\ln^4|  \lambda x- z|\over 1+|z|^{2+\beta}}dz& =\int\limits_{\mathbb R^2}  {\ln^4|  z|\over 1+|\lambda x- z|^{2+\beta}}dz\\
 &=\int\limits_{\{|z|\le1\}}  {\ln^4|  z|\over 1+|\lambda x- z|^{2+\beta}}dz+\int\limits_{\{2|x|\ge |z|\ge1\}}  {\ln^4|  z|\over 1+|\lambda x- z|^{2+\beta}}dz\\ &+\int\limits_{\{ |z|\ge1\}\cap\{|z|\ge 2|x|\}}  {\ln^4|  z|\over 1+|\lambda x- z|^{2+\beta}}dz .
 \end{aligned}$$
  Now,
  $$ \int\limits_{\{|z|\le1\}}  {\ln^4|  z|\over 1+|\lambda x- z|^{2+\beta}}dz\le \int\limits_{\{|z|\le1\}}  {\ln^4|  z| }dz\le c.$$
  $$\begin{aligned}\int\limits_{\{2|x|\ge |z|\ge1\}}  {\ln^4|  z|\over 1+|\lambda x- z|^{2+\beta}}dz&\le \ln^4( 2|x|)\int\limits_{\mathbb R^2}  {1\over 1+|\lambda x- z|^{2+\beta}}dz\\ &=\ln^4( 2|x|)\int\limits_{\mathbb R^2}  {1\over 1+|  z|^{2+\beta}}dz\le c\ln^4( 2|x|). \end{aligned}$$
  $$\begin{aligned}\int\limits_{\{ |z|\ge1\}\cap\{|z|\ge 2|x|\}}  {\ln^4|  z|\over 1+|\lambda x- z|^{2+\beta}}dz&\le \int\limits_{\{ |z|\ge1\}\cap\{|z|\ge 2|x|\}}  {\ln^4|  z|\over 1+2^{-(2+\beta)}| z|^{2+\beta}}dz\\ & \le \int\limits_{\mathbb R^2}  {\ln^4|  z|\over 1+2^{-(2+\beta)}| z|^{2+\beta}}dz\le c,\end{aligned}$$
  because in the last case
  $$|z-\lambda x|\ge |z|-\lambda|x|\ge |z|\(1-\frac \lambda 2\)\ge {|z|\over 2}\ \hbox{if}\ {\lambda\le1}.$$
  Collecting all the previous estimates, \eqref{s1} follows.
  \bigskip
  
  Let us prove \eqref{s2}.
  First of all, by scaling $\lambda z-x=y$ we get
 $$\begin{aligned}  \int\limits_{\mathbb R^2}  {\ln^4|  x-  \lambda z|\over 1+|z|^{2+\beta}}dz& =\lambda^\beta \int\limits_{\mathbb R^2}  {\ln^4|  y|\over \lambda^{2+\beta}+|y+x|^{2+\beta}}dy\\
 &=\int\limits_{\{|y|\le1\}\cap\{|y|\le|x|/2\} }  {\lambda^\beta\ln^4|  y|\over \lambda^{2+\beta}+|y+x|^{2+\beta}}dy+  \int\limits_{\{|x|/2\le |y|\le1\}  }  {\lambda^\beta\ln^4|  y|\over \lambda^{2+\beta}+|y+x|^{2+\beta}}dy\\
 &+ \int\limits_{\{|y|\ge1\}\cap\{|y|\ge2|x| \} }  {\lambda^\beta\ln^4|  y|\over \lambda^{2+\beta}+|y+x|^{2+\beta}}dy+  \int\limits_{\{2|x| \ge |y|\ge1\}  }  {\lambda^\beta\ln^4|  y|\over \lambda^{2+\beta}+|y+x|^{2+\beta}}dy
 \end{aligned}$$
 Now,
 $$\begin{aligned}
  \int\limits_{\{2|x| \ge |y|\ge1\}  }  {\lambda^\beta\ln^4|  y|\over \lambda^{2+\beta}+|y+x|^{2+\beta}}dy&\le  \ln^4(2|x|)
   \int\limits_{\{2|x| \ge |y|\ge1\}  }  {\lambda^\beta \over \lambda^{2+\beta}+|y+x|^{2+\beta}}dy\\ &\le \ln^4(2|x|) \int\limits_{\mathbb R^2 }  {\lambda^\beta \over \lambda^{2+\beta}+|y+x|^{2+\beta}}dy\\ &=\ln^4(2|x|) \int\limits_{\mathbb R^2 }  {\lambda^\beta \over \lambda^{2+\beta}+|y |^{2+\beta}}dy\\ &=
   \ln^4(2|x|) \int\limits_{\mathbb R^2 }  {1 \over 1+|y |^{2+\beta}}dy\le c\ln^4(2|x|).
   \end{aligned}$$
  $$\begin{aligned}
  \int\limits_{\{|y|\ge1\}\cap\{|y|\ge2|x| \} }  {\lambda^\beta\ln^4|  y|\over \lambda^{2+\beta}+|y+x|^{2+\beta}}dy &\le \lambda^\beta
   \int\limits_{\{|y|\ge1\} } { \ln^4|  y|\over 2^{-(2+\beta)} |y |^{2+\beta}}dy\le c,
   \end{aligned}$$
because in this case
  $$|y+ x|\ge |y|- |x|\ge   {|y|\over 2}.$$
 $$\begin{aligned}   \int\limits_{\{|x|/2\le |y|\le1\}  }  {\lambda^\beta\ln^4|  y|\over \lambda^{2+\beta}+|y+x|^{2+\beta}}dy 
&\le \ln^4(|x|/2)  \int\limits_{\{|x|/2\le |y|\le1\}  }  {\lambda^\beta \over \lambda^{2+\beta}+|y+x|^{2+\beta}}dy \\
   &\le \ln^4(|x|/2)  \int\limits_{\mathbb R^2  }  {\lambda^\beta \over \lambda^{2+\beta}+|y+x|^{2+\beta}}dy \\
   &=\ln^4(|x|/2)  \int\limits_{\mathbb R^2  }  {\lambda^\beta \over \lambda^{2+\beta}+|y |^{2+\beta}}dy\\
  &=\ln^4(|x|/2)  \int\limits_{\mathbb R^2  }  {1 \over 1+|y |^{2+\beta}}dy
 \le c\ln^4(|x|/2).
  \end{aligned}$$
  Finally, 
  $$\begin{aligned}\int\limits_{\{|y|\le1\}\cap\{|y|\le|x|/2\} }  {\lambda^\beta\ln^4|  y|\over \lambda^{2+\beta}+|y+x|^{2+\beta}}dy\le \int\limits_{\{|y|\le1\}\cap\{|y|\le|x|/2\} }  {\lambda^\beta\ln^4|  y|\over \lambda^{2+\beta}+2^{-(2+\beta)}| x|^{2+\beta}}dy
   \end{aligned}$$
   because in this case
  $$|y+ x|\ge |x|- |y|\ge   {|x|\over 2}.$$
  If $|x|\ge2$ then
  $$\begin{aligned}  \int\limits_{\{|y|\le1\}\cap\{|y|\le|x|/2\} }  {\lambda^\beta\ln^4|  y|\over \lambda^{2+\beta}+2^{-(2+\beta)}| x|^{2+\beta}}dy&\le  \int\limits_{\{|y|\le1\}\cap\{|y|\le|x|/2\} }    \lambda^\beta \ln^4|  y | dy\\ & \le\lambda^\beta\int\limits_{\{|y|\le1\}  }     \ln^4|  y | dy\le c.
   \end{aligned}$$
   and if $|x|\le2$ then
  $$\begin{aligned}  &\int\limits_{\{|y|\le1\}\cap\{|y|\le|x|/2\} }  {\lambda^\beta\ln^4|  y|\over \lambda^{2+\beta}+2^{-(2+\beta)}| x|^{2+\beta}}dy \le {\lambda^\beta \over  2^{-(2+\beta)}| x|^{2+\beta}}\int\limits_{\{|y|\le1\}\cap\{|y|\le|x|/2\} }   \ln^4|  y| dy\\ &\le {\lambda^\beta \over  2^{-(2+\beta)}| x|^{2+\beta}}\int\limits_{ \{|y|\le|x|/2\} }   \ln^4|  y| dy \le
  {c\over | x|^{2+\beta}}\int_0^{|x|\over2}r\ln^4rdr\\
 & = {c\over | x|^{2+\beta}}(| x|/2)^2\(\frac12\ln^4(| x|/2)-\ln^3(| x|/2)+\frac32\ln^2(| x|/2)-\frac32\ln(| x|/2)+\frac34 \).
   \end{aligned}$$
  Collecting all the previous estimates, \eqref{s2} follows.
  \end{proof}

\section{A contraction mapping argument and the proof of the main theorem}\label{quattro}

First of all we point out  that $\mathbf W_{\la}+\boldsymbol \phi_{\la}$ is a solution to \eqref{p} if and only if
  $\boldsymbol\phi_{\la}$ is a solution of the problem
\begin{equation}\label{L2}
 \mathcal{L}_{\la}(\boldsymbol\phi )=\mathcal{N}_{\la}(\boldsymbol\phi )+ \mathcal{S}_{\la}(\boldsymbol\phi )+\mathcal{R}_{\la} \  \hbox{in}\  \Omega \\\\
\end{equation}
where the error term $\mathcal{R}_\la$ is defined in \eqref{rla},
the linear operator $\mathcal{L}_{\la}$ is defined in \eqref{lla}, the higher order linear operator $\mathcal{S}_{\la}(\boldsymbol\phi )$
is defined as
\begin{align}\label{sla}
&\mathcal{S}_{\la}(\boldsymbol\phi ):=\(S^1_\la(\boldsymbol\phi ),\dots,S^k_\la(\boldsymbol\phi )\),\ \hbox{where}\nonumber\\ &
 S^i_\la(\boldsymbol\phi ):=\(|x|^{\al_i-2}e^{w_i}-2\la e^{W^i_\la}\) \phi^i+\sum\limits_{j=1\atop j\not=i}^k\(-{1\over2} |x|^{\al_j-2}e^{w_j}+\la e^{W^j_\la}  \)\phi^j,\ i=1,\dots,k \end{align}
and the higher order term $\mathcal{N}_{\la}$ is defined as
\begin{align}\label{nla}
&\mathcal{N}_{\la}(\boldsymbol\phi ):=\(N^1_\la(\boldsymbol\phi ),\dots,N^k_\la(\boldsymbol\phi )\),\ \hbox{where}\nonumber\\ &
 N^i_\la(\boldsymbol\phi ):=-2\la e^{W^i_\la}\(e^{\phi^i}+1+\phi^i\)+\la\sum\limits_{j=1\atop j\not=i}^ke^{W^j_\la}\(e^{\phi^j}+1+\phi^j\),\ i=1,\dots,k. \end{align}

\begin{prop}\label{resto}
There exists $p_0>1,$  $\la_0>0$ and $R_0>0$ such that for any $p\in(1,p_0 ),$   $\la\in(0,\la_0)$ and $R\ge R_0$ we have
such that
  for any $\lambda \in (0,\lambda_0)$ there exists a unique solution $\boldsymbol\phi_\lambda =\(\phi_\la^1,\dots,\phi_\la^k\)\in  H_e^k$ (see \eqref{hek}) to the system
\begin{equation}\label{eqrid}
  \Delta (W^i_\lambda+ \phi^i_\lambda)+2\la  e^{W^i_\lambda+ \phi^i_\lambda}-\la\sum\limits_{j=1\atop j\not=i}e^{W^j_\la+\phi^j_\la}=0\ \hbox{in}\ \Omega, \phi^i=0\ \hbox{on}\ \partial\Omega \ \hbox{for}\  i=1,\dots,k.
 \end{equation}
   and
$$\|\phi_\lambda\|\le R\la^{ {1\over2^k}{2-p\over p}}|\ln\lambda| $$
for some $\eps>0.$
\end{prop}

\begin{proof}

 As a consequence of Proposition \ref{inv}, we conclude that $\boldsymbol\phi$
is a solution to \eqref{eqrid} if and only if it is a fixed point
for the operator $ \mathcal{T}_\lambda:
\mathbf H_k\to\mathbf H_k,$ defined by
$$T_\lambda(\boldsymbol\phi)= \left( \mathcal{L}_\lambda\right)^{-1}
\left(\mathcal{N}_\lambda(\boldsymbol\phi)+ \mathcal{S}_{\la}(\boldsymbol\phi ) +\mathcal{R}_\lambda\right),$$
where $\mathcal{L}_\lambda$, $ \mathcal{S}_{\la}$, $\mathcal{N}_\la$ and $\mathcal{R}_\la$  are defined  in \eqref{rla},  \eqref{sla}, \eqref{nla}  and \eqref{rla}, respectively.

 Let   us introduce the ball $B_{\la,R}:=\left\{\boldsymbol\phi\in\mathbf H_k\ :\ \|\boldsymbol\phi\|\le R\lambda^{ {1\over2^k}{2-p\over p} }\right\}$. We will show that $ T_\lambda:B_{\la,R}\to B_{\la,R}$ is a contraction mapping
provided $\la$ is small enough and $R$ is large enough.

\medskip {\em Let us prove that $T_\la$ maps  the ball $B_{\la,r}$ into itself,  i.e.}
\begin{equation}\label{c2.1}
\|\boldsymbol\phi\|\le R\lambda^{ {1\over2^k}{2-p\over p}}|\ln\lambda|\ \Longrightarrow\
\left\|\mathcal{T}_\lambda(\boldsymbol\phi)\right\|\le R\lambda^{ {1\over2^k}{2-p\over p}}|\ln\lambda|.
\end{equation}

By     Lemma \ref{B2} (where we take $h=\mathcal{N}_\lambda(\phi)+ \mathcal{R}_\lambda$), by \eqref{B21}, by Lemma \ref{sla-er} and by Lemma  \ref{errore} we deduce that:
\begin{align*}
\left\|\mathcal{T}_\lambda( \boldsymbol\phi)\right\|&  \le c|\ln\la|  \left(\left\|\mathcal{N}_\lambda (\boldsymbol\phi)\right\|_p +\left\|\mathcal{S}_\lambda (\boldsymbol\phi)\right\|_p  +\left\| \mathcal{R}_\lambda\right\|_p\right)
\\   & \le c|\ln\lambda|\(\lambda^{ 2^{k-1}{1-pr \over
pr}} \|\boldsymbol\phi\|^2+ \la^{{1\over 2^k}{2-p\over p}} \|\boldsymbol\phi\|  + \lambda^{  {1\over2^k}{2-p\over p}}\)
\\   &  \le  R \lambda^{  {1\over2^k}{2-p\over p}}
\end{align*}
provided $r$ and $p$ are close enough to 1,    $R$  is suitable large and $\la $  is  small enough. That proves \eqref{c2.1}.

\medskip {\em Let us prove that $T_\la$ is a contraction mapping,  i.e. there exists $L>1$ such that}
\begin{equation}\label{c2.2}
\|\phi\|\le R \lambda^{{1\over2^k}{2-p\over p}}|\log\rho|\ \Longrightarrow\
\left\|\mathcal{T}_\lambda(\boldsymbol\phi_1)-\mathcal{T}_\lambda(\boldsymbol\phi_2)\right\|\le
L \|\phi_1-\phi_2\|.
\end{equation}

By     Lemma \ref{B2} (where we take $\psi=\mathcal{N}_\lambda(\boldsymbol\phi_1 )-\mathcal{N}_\lambda(\boldsymbol\phi_2 ) $) and by  \eqref{B22}, we deduce that:
\begin{align*}
&\left\|\mathcal{T}_\lambda(\boldsymbol\phi)\right\|   \le c|\ln\la| \(\left\|\mathcal{N}_\lambda(\boldsymbol\phi_1 )-\mathcal{N}_\lambda(\boldsymbol\phi_2 )\right\|_p
+\left\|\mathcal{S}_\lambda(\boldsymbol\phi_1 -\boldsymbol\phi_2 )\right\|_p \)
\le\\ &
  \le c|\ln\la|  \( \lambda^{  2^{k-1}{1-pr \over
pr}}
  \|\boldsymbol\phi_1-\boldsymbol \phi_2\| \( \|\boldsymbol\phi_1\|+\|\boldsymbol \phi_2\|\)+\la^{{1\over 2^k}{2-p\over p}} \|\boldsymbol\phi_1-\boldsymbol \phi_2\| \)  \\   &
  \le L  \sum\limits_{i=1}^k    \|\phi^i_1-\phi^i_2\| =L\|\boldsymbol\phi_1-\boldsymbol\phi_2\| \ \hbox{for some} \ L<1, \end{align*}
provided $r$ and $p$ are close enough to 1,    $R$  is suitable large and $\la $  is  small enough. That proves \eqref{c2.2}.

\end{proof}
\begin{lemma}
\label{sla-er} Let $  \mathcal{S}_\la$ as in \eqref{sla}.
There exists $p_0>1$ and $\la_0>0$ such that for any $p\in(1,p_0 )$ and $\la\in(0,\la_0)$ we have
$$\left\|\mathcal{S}_{\la}(\boldsymbol\phi )\right\|_p=O\(\la^{{1\over 2^k}{2-p\over p}} \|\boldsymbol\phi\|
\)$$\end{lemma}
\begin{proof}
We have that
\begin{align*}
&\left\|\mathcal{S}_{\la}(\boldsymbol\phi )\right\|_p =\sum\limits_{i=1}^k\left\| {S}^i_{\la}(\boldsymbol\phi )\right\|_p \\ &=
O\(\sum\limits_{i=1}^k\left\| \(|x|^{\al_i-2}e^{w_i}-2\la e^{W^i_\la}\) \phi^i \right\|_p \)\ \hbox{(we use H\"older's inequality with ${ 1\over q}+{ 1\over s}=1$)}\\ &
 =O\(\sum\limits_{i=1}^k\left\| \(|x|^{\al_i-2}e^{w_i}-2\la e^{W^i_\la}\) \right\|_{pq} \left\| \phi^i \right\|_{ps} \)\ \hbox{(we use estimate \eqref{er1.2})}\\ &
=O\(\la^{{1\over 2^k}{2-p\over p}}\sum\limits_{i=1}^k \left\| \phi^i \right\|  \) =O\(\la^{{1\over 2^k}{2-p\over p}} \|\boldsymbol\phi\|
\) ,
\end{align*}
which proves the claim. \end{proof}

  \begin{lemma}\label{B2}
    There exists $s_0>1$ and $\la_0>0$   such that for any $p>1,$ $r>1$ with $pr\in(1,s_0)$ and  $\la\in(0,\la_0)$  we have
  for any $\phi,\phi_1,\phi_2\in\{\phi\in \mathrm{H}^1_0(\Omega) \ :\ \|\phi\|\le1\}$
 \begin{equation}\label{B21}
 \left\|\mathcal N_\la(\boldsymbol\phi)\right\|_p=O\( \lambda^{ 2^{k-1}{1-pr \over pr}}\|\boldsymbol\phi \|^2\)\end{equation}
  and
    \begin{equation}\label{B22}
\left\|\mathcal N_\la(\boldsymbol\phi_1)-\mathcal N_\la(\boldsymbol\phi_2)\right\|_p=O\(
 \lambda^{  2^{k-1}{1-pr \over pr}}
   \|\boldsymbol\phi_1-\boldsymbol\phi_2\|\(\|\boldsymbol\phi_1\|+\|\boldsymbol\phi_2\|\)\) .\end{equation}
  \end{lemma}
\begin{proof}
For any $i=1,\dots,k$ set
  $$\texttt{N}^i_\la(\phi):=\la e^{W^i_\la}\(e^\phi+1+\phi\).$$
  By the definition of $\mathcal{N}_\lambda$ in \eqref{nla} we immediately deduce that
  \begin{equation}\label{D21}
\left\|\mathcal{N}_\lambda (\boldsymbol\phi)\right\|_p   =   \sum\limits_{i=1}^k\left\| {N}^i_\lambda (\boldsymbol\phi)\right\|_p =
  O\(\sum\limits_{i=1}^k\left\| \texttt{N}^i_\lambda (\boldsymbol\phi)\right\|_p \)
\end{equation}
  and
 \begin{equation}\label{D22}
\left\|\mathcal{N}_\lambda(\boldsymbol\phi_1 )-\mathcal{N}_\lambda(\boldsymbol\phi_2 )\right\|_p
=\sum\limits_{i=1}^k \left\| {N}^i_\la(\phi^i_1)- {N}^i_\la(\phi^i_2)\right\|_p
= O\(\sum\limits_{i=1}^k \left\| \texttt{N}^i_\la(\phi^i_1)-\texttt{N}^i_\la(\phi^i_2)\right\|_p\).\end{equation}
  We are going to prove that there exist some positive constants $c_i$ such that
   \begin{equation}\label{C21}
 \left\|\texttt{N}^i_\la(\phi)\right\|_p=O\(  e^{c_i\|\phi \|^2}\lambda^{ 2^{k-1}{1-pr \over pr}}\|\phi \|^2\)\end{equation}
  and
    \begin{equation}\label{C22}
\left\|\texttt{N}^i_\la(\phi_1)-\texttt{N}^i_\la(\phi_2)\right\|_p=O\(
  e^{c_i(\|\phi_1\|^2+\|\phi_2\|^2)}\lambda^{  2^{k-1}{1-pr \over
pr}}
   \|\phi_1-\phi_2\|(\|\phi_1\|+\|\phi_2\|)\) .\end{equation}
   Estimate \eqref{B21} follows by \eqref{D21}  and   \eqref{C21} since $\|\phi\| \le1$ and estimate \eqref{B22} follows by \eqref{D22}  and \eqref{C22}, since $  \|\phi_1\|, \|\phi_2\|\le1.$

  \medskip
   Let us prove \eqref{C21} and \eqref{C22}. Since   \eqref{C21} follows by   \eqref{C22} choosing
$\phi_2=0 ,$ we only prove (\ref{C22}). We point
out that
$$\texttt{N}^i_\la (\phi_1)-\texttt{N}^i_\la(\phi_2)=
 {\lambda e^{W^i_\la} \left(e^{\phi_1}-e^{\phi_2}-\phi_1+\phi_2\right) } $$

 By the mean value theorem, we easily deduce that
 $$|e^a-e^b-a+b|\le e^{|a|+|b|}|a-b|(|a|+|b|)\ \hbox{for any }a,b\in\rr.$$
 Therefore,    we have
 \begin{eqnarray}\label{I1p}
 & &\|\texttt{N}^i_\la (\phi_1)-\texttt{N}^i_\la(\phi_2)\|_p=\left(\int\limits_\Omega \lambda^p e^{pW^i_\lambda}\left|e^{\phi_1}-e^{\phi_2}-\phi_1+\phi_2\right|^pdx\right)^{1/p}
 \nonumber\\ & &
 \le c\sum\limits_{j=1}^2
 \left(\int\limits_\Omega  \lambda^p e^{pW^i_\lambda} e^{p|\phi_1|+p|\phi_2|}|\phi_1-\phi_2|^p|\phi_j|^p dx\right)^{1/p}  \nonumber\\ & &\qquad \hbox{(we use H\"older's inequality  with ${1\over r}+{1\over s}+{1\over t}=1$ )}
 \nonumber\\ & &
 \le c\sum\limits_{j=1}^2
 \left(\int\limits_\Omega  \lambda^{pr} e^{prW^i_\lambda} dx\right)^{1/(pr)}
 \left(\int\limits_\Omega e^{ps|\phi_1|+ps|\phi_2|}dx\right)^{1/(ps)}
 \left(\int\limits_\Omega |\phi_1-\phi_2|^{pt}|\phi_j|^{pt} dx\right)^{1/(pt)}
  \nonumber\\ & &\qquad\hbox{(we use Lemma \ref{tmt})}
\nonumber\\ & &\
 \le c\sum\limits_{j=1}^2
 \left(\int\limits_\Omega  \lambda^{pr} e^{prW^i_\lambda} dx\right)^{1/(pr)}
  e^{(ps)/(8\pi)(|\phi_1|^2+|\phi_2|^2)}
  \|\phi_1-\phi_2\|\|\phi_j\|.\end{eqnarray}
We have to estimate
   \begin{align*}
  \int\limits_\Omega  \lambda^{pr} e^{prW^i_\lambda(x)} dx&=\sum\limits_{j=1}^k
 \int\limits_{A_j}  \lambda^{pr} e^{prW^i_\lambda(x)} dx,\end{align*}
 where $A_j$ is the annulus defined in \eqref{anelli}.

 If  $j=i$  we get
  \begin{align*}&\int\limits_{A_i} \lambda^{pr}  e^{prW^i_\lambda(x)} dx\ \hbox{(we use \eqref{tetaj})}\\
  &=\delta_i^2\lambda^{pr} \int\limits_{A_i\over\delta_i} e^{pr \[w_i(\delta_iy)+(\alpha_i-2) \ln|\delta_jy|-\ln2\lambda+\Theta_i(y)\]} dy\\ &
  = {\delta_i^{2-2pr} }\int\limits_{A_i\over\delta_i}\( 2\alpha_i^2{|y|^{\alpha_i-2}\over \(1+|y|^{\alpha_i}\)^2}\)^{pr}e^{pr \Theta_i(y) } dy\ \hbox{(we use Lemma \eqref{teta})} \\
  &=O\({\delta_i^{2-2pr} }\)=O\(\lambda^{2^{k-1}(1-pr) }\)\ \hbox{( because $\delta_j\ge\delta_1=O\(\lambda^{2^{k-2}}\)$   and $pr>1 $)}.\end{align*}
 If  $j\not=i$ by \eqref{er5.0} we deduce
   \begin{align*}&\int\limits_{A_j}\lambda^{pr} e^{prW^i_\lambda(x)} dx = o\( \lambda^{ {1\over 2^k} (2-pr) }\).  \end{align*}
Therefore,   estimate \eqref{C22} follows.
\end{proof}

   We recall the following Moser-Trudinger inequality \cite{Moe,Tru},
 \begin{lemma}\label{tmt} There exists $c>0$ such that for any   bounded domain $\Omega$ in $\rr^2$
 $$\int\limits_\Omega e^{4\pi u^2/\|u\|^2}dx\le c|\Omega|,\ \hbox{for any}\ u\in{\rm H}^1_0(\Omega).$$
 In particular,  there exists $c>0$ such that for any $\eta\in\rr$
 $$\int\limits_\Omega e^{\eta u}\le c|\Omega| e^{{\eta^2\over 16\pi}\|u\|^2},\
 \hbox{for any}\ u\in{\rm H}^1_0(\Omega).$$
 \end{lemma}

\begin{proof}[Proof of Theorem \ref{main}]
By Proposition \ref{resto} we have that
$\mathbf{u}_\la=\mathbf{W}_\la+	\boldsymbol\phi_\la $
is a solution to \eqref{p}.

\medskip
 Let us prove \eqref{quantix}. Let $i=1,\dots,k$ be fixed.
By the mean value theorem we deduce
  $$\int\limits_{\Omega}\la e^{u^i_\la(x)}dx=\int\limits_{\Omega}\la e^{W^i_\la(x)+\phi^i_\la(x)}dx=\int\limits_{\Omega}\la e^{W^i_\la}\(1+e^{t\phi^i_\la}\phi^i_\la\)dx=\int\limits_{\Omega}\la e^{W^i_\la}dx+o(1),$$
  since, arguing exactly as in the proof of Lemma \ref{B2},
we get (for some ${1\over p}+{1\over q}+{1\over s}=1$)
\begin{align*}
 \int\limits_{\Omega}\la e^{W^i_\la}e^{t\phi^i_\la}\phi^i_\la dx=O\(\|\la e^{W^i_\la}\|_p\|\la e^{t\phi^i_\la}\|_{q}\|\phi^i_\la\|_{s}\)=o(1) .
\end{align*}
So we only have to estimate
\begin{align*}
& \int\limits_{\Omega}\la e^{W^i_\la }dx=\int\limits_{\Omega}\la e^{Pw_i(x)-{1\over2}\sum\limits_{j=1\atop j\not=i}^kP w_j(x)}dx\ \hbox{(we use \eqref{anelli})} \nonumber\\
&  =
 \sum\limits_{r=1}^k \int\limits_{A_r}\la e^{Pw_i(x)-{1\over2}\sum\limits_{j=1\atop j\not=i}^kP w_j(x)}dx\nonumber\\
&  =   \int\limits_{A_i}|x|^{\al_i-2}e^{w_i(x)}dx+ \int\limits_{A_i}\(\la e^{Pw_i(x)-{1\over2}\sum\limits_{j=1\atop j\not=i}^kP w_i(x)}-|x|^{\al_i-2}e^{w_j(x)}\)dx
\nonumber\\ & +\sum\limits_{r=1\atop r\not=j}^k \int\limits_{A_r}\la e^{Pw_i(x)-{1\over2}\sum\limits_{j=1\atop j\not=i}^kP w_j(x)}dx\ \hbox{(we apply \eqref{er1.3} and \eqref{er5.13})} \nonumber
\\ &=  \int\limits_{A_i}|x|^{\al_i-2}e^{w_i(x)}dx+o(1)\ \hbox{(we scale $x=\de_iy$)} \nonumber
\\ & = \int\limits_{A_i\over\de_i}|y|^{\al_i-2}e^{w^{\al_i}(y)}dy+o(1)\ \hbox{(because of \eqref{mass})} \nonumber\\ &= 4\pi\al_i+o(1)  .\end{align*}
Here we used the following result of Chen-Li \cite{cl}
\begin{equation}\label{mass}
\int\limits_{\rr^2} |y|^{\al-2}e^{w^\al(y)}dy=2\al^2\int\limits_{\rr^2} {|y|^{\al-2} \over\(1+|y|^\al\)^2}dy=4\pi\al.
\end{equation}
That concludes the proof.
\end{proof}

\section{Appendix}
We have the following result.

\begin{thm}
\label{esposito}  Assume $\alpha=2^i$ for some integer $i\ge1.$
If $\phi  $  satisfies
\begin{equation}\label{even}\phi(y)=\phi(\Re_k y)\ \hbox{for any     $y\in\rr^2$,}\quad \hbox{where}\ \Re_k:=\(\begin{matrix}\cos{\pi\over k}&\sin{\pi\over k}\\
-\sin{\pi\over k}&\cos{\pi\over k}\\ \end{matrix}\)\end{equation} and
solves the equation
\begin{equation}\label{l1}
-\Delta \phi =2\alpha^2{|y|^{\alpha-2}\over (1+|y|^\alpha)^2}\phi\ \hbox{in}\ \rr^2,\quad \int\limits_{\rr^2}|\nabla \phi(y)|^2dy<+\infty,
\end{equation}
then there exists $\gamma \in\rr$ such that
$$\phi(y)=\gamma   {1-|y|^\al\over 1+|y|^\al}.$$
\end{thm}
\begin{proof}
Del Pino-Esposito-Musso in \cite{dem} proved that all the bounded solutions to \eqref{l1} are a linear combination of the following functions (which are written in polar coordinates)
$$\phi_0(y):   ={1-|y|^\alpha\over 1+|y|^\alpha} ,\  \phi_1(y):={ |y|^{\alpha\over 2}\over 1+|y|^\alpha} \cos{\alpha\over2}\theta
,\
\phi_2(y):={ |y|^{\alpha\over 2}\over 1+|y|^\alpha} \sin{\alpha\over2}\theta.
$$
We observe that $\phi_0$   always satisfies \eqref{even}, while if $\alpha=2^i$ for some integer $i\ge1 $ the functions $\phi_1$ and $\phi_2$ do not satisfy \eqref{even}.
In \cite{gp} it was proved  that any solution $\phi$ of \eqref{l1} is actually a bounded solution. That concludes  the proof.
\end{proof}

 For any $\alpha\ge2$  let us consider the Banach spaces
\begin{equation}\label{ljs}
\mathrm{L}_\alpha(\rr^2):=\left\{u \in {\rm W}^{1,2}_{loc}(\rr^2)\ :\  \left\|{|y|^{\alpha-2\over 2}\over 1+|y|^\alpha}u\right\|_{\mathrm{L}^2(\rr^2)}<+\infty\right\}\end{equation}
 and
\begin{equation}\label{hjs}\mathrm{H}_\alpha(\rr^2):=\left\{u\in {\rm W}^{1,2}_{loc}(\rr^2) \ :\ \|\nabla u\|_{\mathrm{L}^2(\rr^2)}+\left\|{|y|^{\alpha-2\over 2}\over 1+|y|^\alpha}u\right\|_{\mathrm{L}^2(\rr^2)}<+\infty\right\},\end{equation}
 endowed with the norms
$$\|u\|_{\mathrm{L}_\alpha}:= \left\|{|y|^{\alpha-2\over 2}\over 1+|y|^\alpha}u\right\|_{\mathrm{L}^2(\rr^2)}\ \hbox{and}\ \|u\|_{\mathrm{H}_\alpha}:= \(\|\nabla u\|^2_{\mathrm{L}^2(\rr^2)}+\left\|{|y|^{\alpha-2\over 2}\over 1+|y|^\alpha}u\right\|^2_{\mathrm{L}^2(\rr^2)}\)^{1/2}.$$
\begin{prop}\label{compact}
The embedding $i_\al:\mathrm{H}_\alpha(\rr^2)\hookrightarrow\mathrm{L}_\alpha(\rr^2)$
is compact.
\end{prop}
\begin{proof}
  See \cite{gp}.
\end{proof}

\end{document}